\documentclass[10pt,oneside,a4paper,reqno]{amsart}

\usepackage[utf8]{inputenc}
\usepackage[hmargin=2cm,vmargin=2cm]{geometry}
\usepackage{parskip}
\usepackage[foot]{amsaddr}

\usepackage[english]{babel}
\usepackage{enumitem}
\usepackage{verbatim}
\usepackage{mathrsfs}

\usepackage{hyperref}
\hypersetup{colorlinks=true,allcolors=blue}

\usepackage{mathtools}
\mathtoolsset{showonlyrefs}
\usepackage{amssymb}
\usepackage{amsthm}
\usepackage{bbm}
\usepackage{aligned-overset}
\usepackage{xcolor}

\newtheorem{theorem}{Theorem}[section]
\newtheorem{lemma}[theorem]{Lemma}
\newtheorem{assumption}[theorem]{Assumption}
\newtheorem{proposition}[theorem]{Proposition}
\newtheorem{corollary}[theorem]{Corollary}
\newtheorem{problem}[theorem]{Problem}
\theoremstyle{definition}
\newtheorem{definition}[theorem]{Definition}
\theoremstyle{remark}
\newtheorem{remark}[theorem]{Remark}

\newcommand{\cB}{{\mathcal B}}

\newcommand{\cF}{{\mathcal F}}

\newcommand{\cS}{{\mathcal S}}

\newcommand{\bbC}{{\mathbb C}}
\newcommand{\bbN}{{\mathbb N}}
\newcommand{\bbP}{{\mathbb P}}
\newcommand{\bbR}{{\mathbb R}}

\renewcommand{\P}{{\mathbb P}}
\renewcommand{\d}{{\rm{d}}}
\newcommand{\dx}{\mathrm{d} x}
\newcommand{\dt}{\mathrm{d} t}

\newcommand{\R}{\mathbb{R}}
\newcommand{\K}{\mathbb{K}}
\newcommand{\C}{\mathbb{C}}
\newcommand{\N}{\mathbb{N}}

\DeclareMathOperator{\Imag}{\mathrm{Im}}

\DeclareMathOperator{\esssup}{ess\,sup}

\DeclarePairedDelimiter{\abs}{\lvert}{\rvert}
\DeclarePairedDelimiter{\norm}{\lVert}{\rVert}
\DeclarePairedDelimiter{\cur}{\{}{\}}
\DeclarePairedDelimiter{\bra}{(}{)}
\DeclarePairedDelimiter{\sqr}{[}{]}

\newcommand{\PP}[2][]{\mathbb{P}\, \sqr[#1]{#2}}
\newcommand{\EE}[2][\big]{\mathbb{E} \sqr[#1]{#2}}

\title{The Schr\"odinger equation with fluctuating nonlinearity in the energy space}
\author{Max Sauerbrey$^1$}
\address{$^1$Max Planck Institute for Mathematics in the Sciences\\
Inselstr. 22 \\ 04103 Leipzig \\ Germany.} \email{maxsauerbrey97@gmail.com}
\author{Joris van Winden$^{2,3}$}
\address{$^2$Mathematical Institute, Leiden University\\ 
Einsteinweg 55\\ 2333 CC Leiden \\ The Netherlands.}
\email{j.van.winden@math.leidenuniv.nl}
\address{$^3$Delft Institute of Applied Mathematics, Delft University of Technology, Mekelweg 4, 2628 CD Delft, The Netherlands}
\date{September 9, 2026}
\thanks{JvW acknowledges support from a DIAM fast-track scholarship and from the NWO grant VI.C.242.087, `Patterns in Random Media'.}

\begin{document}

    \begin{abstract}We study nonlinear Schr\"odinger equations with nonlinear Stratonovich noise
    \begin{equation*}
    \d u\,=\, i\bigl[ \Delta u \,+\, \lambda|u|^{p-1}u\bigr] \,  \d t \,+\,i|u|^{(q-1)/2}u\circ \d {W},
    \end{equation*} in their energy space $H^1(\R^d;\C)$. By combining the stochastic Strichartz estimates derived in [Potential Anal. 41 (2014), pp.\ 269--315] with the approach from [Ann.\ Inst.\ H.\ Poincar\'e Phys.\ Th\'eor.\ 46 (1987), pp.\ 113--129] we obtain local well-posedness for all energy-subcritical nonlinearities $p,q\in [1, 1+4/(d-2)_+)$ together with a corresponding blow-up alternative. For a linear multiplicative noise $q=1$, a real-valued noise   $W$
 and a defocusing  nonlinearity $\lambda\le 0$, we check this blow-up condition using a bound on the energy, resulting in the global well-posedness of the equation. If both nonlinearities are mass-subcritical, i.e., $p,q\in [1, 1+4/d)$, we provide an improved blow-up criterion involving the $L^2(\R^d;\C)$-norm. Using the conservation of mass for real-valued $W$, we obtain global well-posedness  also in this case.
    Compared to previous results on stochastic nonlinear Schr\"odinger equations, we thereby improve the range of exponents $p$ and $q$ and the spatial regularity assumption on the noise.
\end{abstract}

\maketitle

\section{Introduction}\label{Sec:Intro}

We consider the stochastic nonlinear Schr\"odinger equation with fluctuating nonlinearity interpreted in the Stratonovich sense
\begin{equation}\label{Eq100}\tag{SNLS}
 \d u \,=\, i\bigl[\Delta u +\lambda|u|^{p-1}u\bigr]\, \d t\,+\, i|u|^{(q-1)/2}u \circ \d W,\qquad u(0)\,=\, u_0,
\end{equation}
on the whole space $\R^d$.  Here, $\lambda\in \R$ and $p,q\in [1,\infty)$ are parameters and $\frac{\d W}{\d t}$ is a Gaussian process which is white in time and suitably colored in space. 
Although this equation has seen significant interest in recent decades, a complete picture of local well-posedness in $H^1_x$ has been absent, even in the case $q = 1$.

In the current work, we establish a stochastic version of Kato's approach to Schr\"odinger equations \cite{kato} and obtain local well-posedness in the full energy-subcritical range of deterministic as well as stochastic exponents, under optimal assumptions on the noise regularity.

 \subsection{State of the art}
The well-posedness of nonlinear Schr\"odinger equations is often determined by the relation of the nonlinearities  to the Laplacian when rescaled invariantly with respect to the $L_x^2$ and $H_x^1$-norm, respectively, appearing in the typically conserved quantities
\begin{equation}\label{eqn_conserved_qtts}\text{mass}  \qquad \int_{\R^d} |u|^2 \,\d x\qquad \text{and energy} \qquad  \int_{\R^d} \frac{|\nabla u|^2}{2}  - \frac{\lambda |u|^{p+1} }{p+1}\, \d x.
\end{equation}
Accordingly, the relations  $p < 1 + 4/d$, $p=1+4/d$ and $p>1+4/d$ refer to the deterministic nonlinearity in \eqref{Eq100}  being mass-subcritical, critical or supercritical respectively.
For energy-criticality the situation is similar, but with $1+4/d$ replaced by $1+4/(d-2)_+$ (with the convention that ${1}/{0}=\infty$).  Since only for non-positive $\lambda$ the energy has a sign, one distinguishes moreover between defocusing ($\lambda\le 0$) and focusing ($\lambda>0$) deterministic nonlinearities. 
Concerning nonlinear Schr\"odinger equations with additive noise, global well-posedness for energy-subcritical, defocusing nonlinearities was shown in \cite{DeBouard_Debussche_H1}, while it leads for focusing supercritical nonlinearities to blow-up of the solution \cite{BlowUp_2002}. Global well-posedness for defocusing, critical nonlinearities in $L_x^2$ and $H_x^1$, respectively, was recently shown in \cite{Oh_Okamoto}. 
Concerning nonlinear Schr\"odinger equations with linear  multiplicative noise, the global in time well-posedness in a restricted mass-subcritical setting   $ p< 1+\min\{4/d, 2/(d-1)\}$ was obtained in the pioneering work \cite{DeBouard_Debussche_1999}. A similarly restricted energy-subcritical range 
\begin{equation}\label{Eq35}
	p \,\in\, [1,1+2/(d-1)_+) \cup [2, 1+4/(d-2)_+)
\end{equation}
 for defocusing Schr\"odinger equations with linear multiplicative noise appears in the aforementioned \cite{DeBouard_Debussche_H1} as well. 
 Blow-up by linear multiplicative noise for focusing supercritical nonlinearities was on the contrary shown in \cite{Blow_up_2005}.
 A study of the blow-up behavior of solutions in the  focusing case  was performed in the recent works \cite{millet2025wellposednessfocusingstochasticnonlinear,millet2026energycriticalstochasticnonlinearschrodinger} for energy-subcritical and critical nonlinearities, respectively.
 Mass-critical defocusing nonlinearities were recently treated in \cite{FX_subcrit_approx,FX_wellposedness} for a noise $W$ which is sufficiently smooth and decaying as $|x|\to\infty$. Under the latter assumption a rescaling transformation of the equation is possible, leading to a random PDE for the variable $\exp(W)u$. This allowed the authors of \cite{BRZ_2014} and \cite{BRZ_2016} to treat mass-subcritical and energy-subcritical, defocusing nonlinearities, respectively, and critical, defocusing nonlinearities in  \cite{zhang_critical}. With applications to open quantum systems in mind, cf.\  Subsection \ref{SS:motivation} below, the noise in these works is taken in a way which turns the mass into a martingale in the time variable. For the case of a purely imaginary noise $W$, such a noise can be shown to lead to improved blow-up and scattering behaviour for focusing nonlinear Schr\"odinger equations \cite{BRZ_nonBlowUp, herr2019scattering,spitz2025regularizationnoiseenergymasscritical}. For Schr\"odinger equations with supercritical nonlinearity, regularization by a non-local, superlinear noise was recently proved in \cite{brzeźniak2024globalposednessergodicresults}.  
 Concerning \eqref{Eq100} with $q>1$, the only works we are aware of are \cite{hornung_thesis,hornung_SNLS}.
 In~\cite{hornung_SNLS}, local well-posedness in $L_x^2$ for mass-subcritical or mass-critical nonlinearities $p,q\in [1, 1+4/d]$ is shown as well as global well-posedness, 
 if additionally

 \begin{equation}\label{Eq36}
 p,q< 1+4/d\qquad \text{and}\qquad 
q\,\le \, 2\,\frac{p-1}{p+1}\frac{4+d(1-p)}{4p+d(1-p)} \,+\, 1.
 \end{equation}
 In~\cite{hornung_thesis}, local well-posedness in $H^1_x$ with nonlinear noise is obtained, for $p,q$ both in the restricted energy-subcritical range given by~\eqref{Eq35}.

 \subsection{Statement of the main result and discussion}
 The following local and global well-posedness of~\eqref{Eq100} in the energy space $H^1_x$ is a consequence of our main results
 Theorem \ref{Thm:WP_H1} and Proposition \ref{prop:global} in combination with  well-established estimates for \eqref{Eq100}, which we recall in Section \ref{Sec_quant}.

 \begin{theorem}\label{thm_intro} On a complete  probability space $(\Omega, \mathfrak{A},\P)$ with filtration $(\mathcal{F}_t)_{t \geq 0}$ satisfying the usual conditions, let 
 	\begin{equation*}
 		u_0 \in L^0_{\Omega}(H^1_x)   \text{ be $\mathcal{F}_0$-measurable}
 		\end{equation*}
 	and
 	\begin{equation}\label{noise_reg_intro}
 	W(t) \,=\,\sum_{k\in \N} \phi_k \beta_k(t),\qquad  (\phi_k)_{k\in \N}\in L^{\infty}(\R^d;\ell^2(\bbC)),
 	\end{equation}
such that
 \begin{equation}\label{noise_reg_intro_extra}(\nabla \phi_k)_{k\in \N}\in L^{\nu}(\R^d;\ell^2(\C^d)),\qquad  
 	\begin{cases}
 	\nu \ge \max\{2,d\}, & d\ne 2,\\
 	\nu>2  , & d=2,
 \end{cases}
 \end{equation}
 for a family $(\beta_k)_{k\in \N}$ of independent $\mathcal{F}$-Brownian motions.
 
  Then, for $p,q\in [1,1+4/(d-2)_+)$  (i.e., for energy-subcritical nonlinearities), there exists a maximal unique local solution to \eqref{Eq100} in $H^1_x$.
   For real-valued $(\phi_k)_{k\in \N}$, the latter exists globally in time if either
   \begin{enumerate}[label=(\alph*)]\item \label{item_A} the noise is linear multiplicative noise and the nonlinearity is defocusing:  $q=1$ and  $\lambda \le 0$, or
   	\item \label{item_B} both nonlinearities are mass-subcritical:
   	$p,q<1+4/d$.
   \end{enumerate} \end{theorem}
Let us briefly comment on how the above advances the well-posedness theory of stochastic nonlinear Schr\"odinger equations compared to previous works:
\begin{enumerate}[label=(\roman*)]
    \item The local well-posedness of Schr\"odinger equations in $H^1_x$ with nonlinear noise (i.e., $q \neq 1$) extends the results from \cite{hornung_thesis} beyond the restriction \eqref{Eq35} to the full subcritical regime.
	\item In the case of linear noise ($q = 1$), we fill the same gap  for $p$. In the situation \ref{item_A} this yields also new local and global well-posedness results compared to~\cite{DeBouard_Debussche_H1} without imposing additional regularity and decay assumptions on the noise (as was done in~\cite{BRZ_2016} using a rescaling transformation).
    \item In the mass-subcritical case \ref{item_B}, we obtain global well-posedness without the second condition of~\eqref{Eq36} (as imposed in~\cite{hornung_SNLS}), which additionally coupled $q$ to $p$.  
\end{enumerate}
Moreover, as a consequence of our results, \cite[Theorems 3.6--3.7]{millet2025wellposednessfocusingstochasticnonlinear} should now hold unconditionally in the full $L^2_x$-critical and intercritical range, regardless of dimension.
Indeed, the authors write `we remark that our results also hold conditionally
(upon the local well-posedness)' in \cite[p.8, ll.35--36]{millet2025wellposednessfocusingstochasticnonlinear}.

Next to our results, we also expect our method to prove local well-posedness laid out in the following subsection to be of interest in the study of other stochastic PDEs.
\subsection{Strategy to prove local well-posedness}\label{SS:strategy_LWP} To obtain local well-posedness of \eqref{Eq100} in $H^1_x$, we establish  a stochastic variant of  Kato's approach \cite{kato} to treat deterministic nonlinear Schr\"odinger equations
by applying the Banach fixed point argument on the set
\begin{equation}\label{Eq105}
\Bigl\{ u \in  L^\infty(0,T;H^1_x)\cap  L^r(0,T; W^{1,p+1}_x) \,\Big|\, 
\|u \|_{L^\infty(0,T;H^1_x)\cap  L^r(0,T; W^{1,p+1}_x) }\le R
\Bigr\}
\end{equation}
(for a suitable choice of $p,r$) equipped with the metric induced by the norm of 
\begin{equation}\label{Eq106}
C([0,T];L_x^2)\cap  L^r(0,T; L^{p+1}_x).
\end{equation}
The advantage of this strategy is that only $C^1$-regularity of the nonlinearity is needed, in contrast to working with the norm topology of \eqref{Eq105}, which would require $C^2$-regularity instead (which would require $p \geq 2$ and $q \geq 2$). This is precisely the source of the gap of exponents in \eqref{Eq100} in the work \cite{DeBouard_Debussche_H1} for dimensions $d\ge 4$. We remark that this idea has also recently proven fruitful in the numerical study of hyperbolic SPDEs  \cite{klioba_veraar}.

The cut-off parameter   $R$ present in \eqref{Eq105} serves two purposes: Firstly, it guarantees that the set~\eqref{Eq105} is complete with respect to the metric of~\eqref{Eq106} by the Banach--Alaoglu theorem. Secondly, it provides a control on $u$ and thereby allows to prove Lipschitz bounds on the nonlinearity for elements of \eqref{Eq105}. Together with Strichartz estimates for the linear Schr\"odinger evolution, this  results in a contraction estimate for sufficiently small times $T>0$. 
In a stochastic setting however, since the fluctuations of the noise $W$  in \eqref{Eq100} may be very large with positive probability, we cannot ensure that the solution $u$ remains $\P$-a.s.\ in a norm-bounded set like \eqref{Eq105} for any fixed time $T>0$. 
A common solution to this problem is to instead work with a suitable truncated nonlinearity, which coincides with the actual one for any $u$ from \eqref{Eq105}. This is exactly the strategy employed in \cite{DeBouard_Debussche_1999,hornung_SNLS} for $L^2_x$-solutions, where, e.g., the deterministic nonlinearity $|u|^{p-1}u$ is replaced by 
\begin{equation}\label{Eq80}
\varphi_R \bigl( \| u\|_{L^r(0,t; L^{p+1}_x)}\bigr)|u|^{p-1}u,\qquad  \varphi_R(s)  = \begin{cases}
	1, &s<  R,\\
	2-s/R , & s\in [R, 2R),
	\\
	0, &s\ge  2R.
\end{cases}
\end{equation}

Indeed, this effectively imposes a bound on $u$ in $L^r(0,T; L^{p+1}_x)$, while not affecting the nonlinearity until some positive stopping time $\tau>0$. Based on the stochastic Strichartz estimates derived in \cite{stoch_strichartz}, this allowed the author of \cite{hornung_SNLS} to prove the existence and uniqueness of local $L^2_x$-solutions even for mass-critical nonlinearities.

It turns out to be quite delicate to apply a truncation in Kato's setting.
The main difficulty is that one seemingly requires a truncation which bounds the stronger norms appearing in \eqref{Eq105} but which is still Lipschitz with respect to the weaker norm \eqref{Eq106}.
Note that the truncation 
\[
\varphi_R\bigl( \| u\|_{L^r(0,t; W^{1,p+1}_x)}\bigr)|u|^{p-1}u,
\]
in the spirit of \eqref{Eq80},
does not have the required properties, since it is not Lipschitz with respect to $L^r(0,T;L^{p+1}_x)$.
We alleviate this issue by precisely analyzing the  argument of \cite{kato}: Firstly, one splits the nonlinearity  into a Lipschitz part $F_1$ and a polynomial-like part $F_2$, so that the Lipschitz nonlinearity $F_1$ does not require any truncation. The estimate on $F_2(u) -F_2(v)$ is then obtained  in terms of $\|u-v \|_{L^r(0,T; L^{p+1}_x)}$ with a growth factor depending on $\|u\|_{ L^\infty(0,T;L^{p+1}_x)} \vee \|v\|_{ L^\infty(0,T;L^{p+1}_x)}$,
and therefore only these norms need to be controlled by our  truncation.
Fortunately, this is possible by truncating the $L^{p+1}_x$-norm pointwise in $t$ (an operation which is Lipschitz with respect to $L^r(0,T;L^{p+1}_x)$).
These observations yield a truncation-based proof of \cite{kato}, which, as we demonstrate in this manuscript, can be made probabilistic using the stochastic  Strichartz estimates of \cite{stoch_strichartz}. 

\subsection{Outlook: The energy-critical case and a problem concerning truncation}
To highlight the subtle role of finding suitable truncation operators when establishing stochastic versions of Kato's approach, we  discuss  possible extensions of our local well-posedness result regarding \eqref{Eq100} with energy-critical nonlinearities. 
In the deterministic setting, local well-posedness in critical spaces was shown  in~\cite{CW_crit} using a version of Kato's approach.
The argument makes use of different function spaces compared to~\eqref{Eq105}--\eqref{Eq106}, and it is unclear whether a suitable truncation operator exists.
More precisely, one could ask the following.

\begin{problem}\label{prob}
    For $d \geq 3$, let $r = 2d/(d-2)$, $q = 2d/(d-2+4/d)$, $\tilde{q} = 2d/(d-4+4/d)$, $X = L^r(0,T;L^{q}_x(\bbR^d))$ and $Y = L^r(0,T;L^{\tilde{q}}_x(\bbR^d))$.
    Does there exist a mapping $\phi \colon X \cap Y \to X \cap Y$ which satisfies the following properties:
    \begin{enumerate}
        \item (identity) $\phi(f) = f$ whenever $\norm{f}_Y \leq 1$,
        \item (truncation) $\norm{\phi(f)}_Y \lesssim 1$,
        \item (Lipschitz) $\norm{\phi(f) - \phi(g)}_X \lesssim \norm{f - g}_X$.
    \end{enumerate}
\end{problem}
We expect that an affirmative answer to Problem \ref{prob} would allow to adapt the proof of \cite{CW_crit} and obtain local well-posedness for \eqref{Eq100} with energy-critical nonlinearities.

\subsection{Applications}\label{SS:motivation}
In addition to their fundamental role in the study of dispersive PDEs, Schr\"odinger equations arise in various applications. In quantum mechanics, the wave function of a particle is described by a linear Schr\"odinger equation usually supplemented with a potential term, which accounts for environmental forces   \cite{griffiths2018introduction}. Nonlinear Schr\"odinger equations arise in the modeling of
 Bose--Einstein condensation, where they are referred to as the Gross--Pitaevskii equation \cite{pitaevskii2016bose}, or in the description of electromagnetic fields traveling through nonlinear media. Indeed, the pulse envelope $U$ around a plane wave of an optical pulse traveling through a disordered optical fiber can be shown to satisfy
 \begin{equation}\label{Eq9_new}
 	 \partial_y U(s,y) \,=\, i\partial_{ss} U(s,y) \,-i\, \Lambda |U(s,y)|^2 U(s,y),
 \end{equation}
 in a moving frame $(s,y)$, cf.\ \cite{boyd2003nonlinear}. The cubic  nonlinearity is due to an intensity-dependent refractive index prescribed by the Kerr effect.
 Corresponding stochastic Schr\"odinger equations arise in the modeling of open quantum systems \cite{Barchielli_Gregoratti_book}, temperature effects on Bose--Einstein condensation  \cite{Review_Stoch_GP} or excited atoms within an optical medium resulting in Raman scattering \cite{agrawal2000nonlinear, essiambre2008capacity}. If the material coefficient $\Lambda$ in \eqref{Eq9_new}  varies moreover quickly and randomly in $y$, the latter equation turns into \eqref{Eq100}, cf.\ \cite{abdullaev2000soliton}. Next to its mathematical interest, we have the latter application in mind  when studying Schr\"odinger equations with fluctuating nonlinearity. 

\subsection{Organization of the manuscript}
The proof strategy laid out in Section \ref{SS:strategy_LWP}
is carried out in Section~\ref{sec_local} resulting in Theorem \ref{Thm:WP_H1}, i.e., the local well-posedness of \eqref{Eq100} as stated in Theorem \ref{thm_intro}. To obtain global solutions for mass-subcritical (deterministic and stochastic) nonlinearities we derive an improved blow-up condition in terms of the $L^2_x$-norm in Section \ref{Sec_APE}, see Proposition \ref{prop:global}.
In the last Section \ref{Sec_quant}, we recall by now standard estimates on the mass and energy of a solution to \eqref{Eq100}, which together with the aforementioned blow-up conditions result in the assertions regarding global well-posedness from Theorem \ref{thm_intro}.

\section{Notation}
Throughout the article we fix a parameter $d\in \N$ denoting the spatial dimension. We let $(\Omega, \mathfrak{A},\bbP )$ be a complete probability space equipped with a filtration $\cF$ satisfying the usual conditions. Accordingly, we let $(\beta_k)_{k\in \N}$ be a family of independent $\cF$-Brownian motions.

We use standard notation for function spaces on the whole space $\R^d$. E.g., for the Lebesgue space of $p$-integrable functions with respect to the Lebesgue measure $\dx$ we write $L^p(\R^d; \K)$ equipped with the norm
\[
\|f\|_{L^p(\R^d; \K)} \,=\, \begin{cases}
\bigl(\int_{\R^d} |f(x)|^p\,\dx\bigr)^{1/p}, & p\in [1,\infty),\\
\esssup_{x\in \R^d} |f(x)|,
& p=\infty,
\end{cases}
\]
 where $\K\in \{\R, \C\}$. For $k\in \N_0$ and $p\in [1,\infty)$, we write $W^{k,p}(\R^d;\K)$ for the Sobolev space consisting of functions $f$ such that
 \begin{equation}\label{Eq17}
 \|f\|_{W^{k,p}(\R^d;\K)}^p \,=\, \sum_{|\alpha|\le k} \|\partial_\alpha f \|_{L^p(\R^d; \K)}^p < \infty,
 \end{equation}
where the above sum runs over multiindices $\alpha\in \N_0^d$ and $\partial_\alpha$ denotes the distributional derivative. For $p=\infty$, we replace the sum in \eqref{Eq17} by a maximum.
Lastly, for the complex-valued spaces, which are used frequently in this manuscript, we abbreviate
\begin{align*}
	L^{p}_x \coloneqq L^{p}(\R^d;\C), \qquad
	W^{k,p}_x \coloneqq W^{k,p}(\R^d;\C), \qquad
	H^k_x \coloneqq W^{k,2}(\R^d;\C), 
\end{align*}
for $k \in \N_0$, $p \in [1,\infty]$ to ease notation. The space of Schwartz functions is denoted by $\mathscr{S}(\R^d)$ and its topological dual, the space of Schwartz distributions, by $\mathscr{S}'(\R^d)$. 
Concerning mostly the nonlinearities in the Schr\"odinger equation, we write $C^k(\R^2;\C)$ for the set of $k$-times continuously differentiable functions from $\R^2$ to $\C$ for $k\in \N$. We remark that while the nonlinearities in the Schr\"odinger equation take complex arguments, we use $\R^2 \eqsim \C$ as domain to stress that we require real rather than complex differentiability.

We employ the usual notation for vector-valued function spaces: If $(\cS,\cB, \mu)$ is a measure space and $X$ is  a  Banach space, we denote by $L^p(\cS; X)$ the Bochner space of strongly measurable $X$-valued functions, carrying the norm
\[
\|f\|_{L^p(\cS; X)}^p \,=\, \int_\cS \|f\|_X^p \d \mu,
\]
for $p\in [1,\infty)$ with the usual modification for $p=\infty$. If $(\cS,\cB, \mu)$ is the underlying probability space $(\Omega, \mathfrak{A},\bbP )$ we also use the shorthand notation $ L^p_{\Omega}(X)$ for $ L^p(\Omega;X)$.  If on the other hand $(\cS,\cB, \mu) = ([0,T],\mathfrak{B}([0,T]), \dt)$ we make use of the notation $L^p(0,T;X)$ instead. 
For the  space of continuous $X$-valued functions we write $C([0,T];X)$ and equip it with the norm
\[
\|f\|_{C([0,T];X)}\,=\, \sup_{t\in [0,T]}\|f(t)\|_X.
\]
For the closed subspace of progressively measurable random variables of $L_\Omega^p( L^q(0,T;X) ) $ and $L_\Omega^p( C([0,T];X) ) $ (or variants thereof) we write
$L_{\mathcal{P}}^p( L^q(0,T;X) )$ and $ L_{\mathcal{P}}^p( C([0,T];X)  )$, respectively. More generally, if $\tau$ is a stopping time, $L_{\mathcal{P}}^p( L^q(0,\tau;X) )$ denotes the set of all progressively measurable $X$-valued processes $f$ such that $\mathbf{1}_{[0,\tau]}f \in L^p_\mathcal{P}( L^q ( [0,\infty) ;X) )$, with corresponding norm. If $\tau$ is $\P$-a.s.\ finite, $ L_{\mathcal{P}}^p( C([0,\tau];X)  )$ is then the subset of processes in $L_{\mathcal{P}}^p( L^\infty(0,\tau;X) )$, which admit $\P$-a.s.\ an $X$-continuous version on $[0,\tau]$. 
If $\cS=\N$ is equipped with the counting measure, we write $\ell^p(X)$ for $L^p(\N; X)$ and $\ell^p$ for $\ell^p(\C)$. We remark that in any of these situations we denote by $L^0$ the set of measurable mappings.

Lastly, for two Banach spaces $X$ and $Y$ we denote the space of bounded linear operators by $L(X,Y)$ and for Hilbert spaces $G,H$ we denote the space of Hilbert--Schmidt operators by $L_2(G,H)$. The intersection of two Banach spaces is as usual equipped with the norm $\|\cdot \|_{X\cap Y} = \|\cdot \|_X + \|\cdot \|_Y$.

\section{Local well-posedness}\label{sec_local}
The aim of this section is to prove local in time well-posedness of stochastic nonlinear   Schr\"odinger equations of the form
\begin{equation}
\label{eq:general_SNLS}\begin{cases}
    \d  u = [i \Delta u +  F(x,u)] \d t + G(u) \d W(t),\\
    u(0)= u_0
	\end{cases}
\end{equation}
for initial data from the energy space $H^1_x$. We remark that, in contrast to the Stratonovich equation \eqref{Eq100} considered in the introduction, the above is interpreted in the sense of It\^o's theory of stochastic integration. Accordingly, we allow for a space dependent  nonlinearity $F$ to account also for possible correction terms when rewriting \eqref{Eq100} in its It\^o-form
\begin{equation}
\label{eq:snls_ito}
 \d u \,=\, i\bigl[\Delta u +\lambda|u|^{p-1}u  \bigr]\, \d t\,+\,\Psi|u|^{q-1} u\,  \d t \,+\,  i|u|^{(q-1)/2}u \, \d W,
\end{equation}
where we have abbreviated
\begin{equation}
    \label{eq:psi_correction}
    	\Psi \,=\, \textstyle \frac{q-1}{8} \sum_{k\in \N} 
	|\phi_k|^2\,-\, \frac{q+3}{8}  \sum_{k\in \N} 
	\phi_k^2.
\end{equation}

Before stating our assumptions, we recall the identification $\C \eqsim \R^2$ and that the notation $C^1(\R^2;\C)$ stresses that we assume differentiability in the real sense. We also remind ourselves that $(\beta_k)_{k\in \N}$ is a sequence of independent $\cF$-Brownian motions. 
\begin{assumption}[Local well-posedness in $H^1$]
    \label{ass:nonlinear} 
    There exist $p \in (1,1 + \frac{4}{(d - 2)_+})$, $C<\infty$, $N\in \N$ and \begin{equation}\label{eqn:nu_exp}
    	\begin{cases}
    		\nu \ge \max\{2,d\}, & d\ne 2,\\
    		\nu>2  , & d=2,
    	\end{cases}
    \end{equation}such that the following holds:
    \begin{enumerate}[label=(\roman*)]
    	\item \label{ass:nonlinear_I}  $u_0\in L_{\Omega}^0(H^1_x)$ is $\cF_0$-measurable.
    	\item \label{ass:nonlinear_II}  We have the decomposition
    	\begin{equation}\label{eqn_f_tilda_equals}
    		F(x,u) \,=\,\sum_{n=1}^N \Psi_n (x) F^{(n)}(u)
    	\end{equation}
    	for functions 
    	$F^{(n)} \in C^1(\R^2;\C)$ with  $F^{(n)}(0) = 0$ as well as
    	\begin{align*}
    		\abs{(F^{(n)})'(x)} &\,\leq \, C \abs{x}^{p-1},\qquad |x|\,\ge\, 1,
    	\end{align*}
	and $\Psi_n \in L_x^\infty$ is such that $\nabla \Psi_n \in L^\nu(\R^d;\C^d)$, for $n\in \{1,\dots ,N\}$.
	\item  \label{ass:nonlinear_G} We have	$G\in C^1(\R^2;\C)$ with  $G(0) = 0$ and
	\begin{align*}
		\abs{G'(x)} &\,\leq \, C \abs{x}^{\frac{p-1}{2}}, \qquad |x| \,\ge\, 1.
	\end{align*}
    \item \label{ass:nonlinear_III} The driving noise admits the expansion
    \begin{equation}\label{Eq15}
    	W(t)\,=\, \sum_{k\in \N} \phi_k \beta_k(t)
    \end{equation}
for a sequence $(\phi_k)_{k\in \N}\in L^\infty(\R^d;\ell^2)$ satisfying additionally $(\nabla \phi_k)_{k\in \N}\in L^{\nu}(\R^d;\ell^2(\C^d))$.
\end{enumerate}
\end{assumption}

\begin{remark} \label{rem:subcritical}
	The  Assumptions \ref{ass:nonlinear}~\ref{ass:nonlinear_I}--\ref{ass:nonlinear_III} express that the nonlinearities  are \emph{subcritical} in terms of scaling. Indeed, Assumption \ref{ass:nonlinear}~\ref{ass:nonlinear_I} specifies that we work in the energy space $H^1_x$ which (neglecting lower order contributions) is preserved under the scaling $(t,x)=(\epsilon^2\tilde{t},\epsilon \tilde{x})$ of the linear operators in \eqref{eq:general_SNLS} if we also rescale $u=\epsilon^{1-\frac{d}{2}}\tilde{u}$.  In these new variables, the SPDE \eqref{eq:general_SNLS} reads
    \begin{equation}\label{Eq84}
    \d  \tilde{u}\, =\, \bigl[i \Delta \tilde{u}\, + \,\epsilon^{\frac{d}{2}+1}F(\epsilon\tilde{x},\epsilon^{1-\frac{d}{2}}\tilde{u})\bigr] \d \tilde{t}\, + \,\epsilon^{\frac{d}{2}}G(\epsilon^{1-\frac{d}{2}}\tilde{u}) \d \tilde{W}(\tilde{t}),
    \end{equation}
    where we also replaced $W=\epsilon\tilde{W}$ to ensure that the noise is again a Wiener process. Consequently, nonlinearities of the form $|F^{(n)}(u)|\sim |u|^p$ and $G(u)\sim|u|^{\frac{p+1}{2}}$ vanish as $\epsilon\searrow 0$ if $p<1+\frac{4}{(d-2)_+}$. We deduce that under Assumption \ref{ass:nonlinear}~\ref{ass:nonlinear_II}--\ref{ass:nonlinear_III}  the dispersive effects due to the linear part of \eqref{eq:general_SNLS} are stronger than the nonlinearities, at least on a microscopic level. This is why the latter are called  energy-subcritical in this case.
 \end{remark}
\begin{remark}
	The previous remark highlights also the optimality of the exponent from \eqref{eqn:nu_exp}, at least for $d\ge 3$, where we can take $\nu=d$. Indeed, the sizes of  $\Psi_n$ in $L^\infty_x$ and $\nabla \Psi_n $ in $L^d(\R^d;\C^d)$ are invariant under the rescaling of space in \eqref{Eq84}, and the same applies to the noise coefficients $(\phi_k)_{k\in \N}$.
\end{remark}
\begin{remark}\label{rem:noise_ex}
	Any $Q$-Wiener process in a Hilbert space ${H}\subset \mathscr{S}'(\R^d)$ admits an expansion of the form \eqref{Eq15} with 
	$\phi_k =\sqrt{\lambda_k}f_k$ for an orthonormal basis  $(f_k)_{k\in \N}$ of $H$ consisting  of eigenfunctions of $Q$ with eigenvalues $(\lambda_k)_{k\in \N}\in \ell^1$, see \cite[Section 2.1]{Liu_Rockner_intro}. In this case, Assumption \ref{ass:nonlinear}~\ref{ass:nonlinear_III} can be read off the eigenfunctions and their corresponding eigenvalues.
	As a second example we consider
	\[
	W(t) \,=\, \eta*\biggl(\sum_{k\in \N}   e_k \beta_k(t)\biggr),
	\]
	for an orthonormal basis $(e_k)_{k\in \N}$ of $L^2(\R^d)$ and a convolution kernel $\eta \in \mathscr{S}'(\R^d)$. In this case  the SPDE \eqref{eq:general_SNLS} has a spatially regularized space-time white noise on the right-hand side when written in differential form. We obtain the expansion \eqref{Eq15} for $\phi_k = \eta*e_k$, so that
	\begin{align*}
	\|\eta\|_{L^2_x}^2 \,&=\, \sup_{x\in \R^d}\sum_{k\in \N} \biggl|\int_{\R^d} 
	\eta(x-y)e_k(y)
	 \,\d y \biggr|^2\,=\, \|\phi\|_{L^\infty(\R^d;\ell^2)}^2,\\
	 \|\nabla \eta\|_{L^2(\R^d; \C^d)}^2 \,&=\, \sup_{x\in \R^d}\sum_{k\in \N} \biggl|\int_{\R^d} 
	 \nabla \eta(x-y)e_k(y)
	 \,\d y \biggr|^2\,=\, \|\nabla\phi\|_{L^\infty(\R^d;\ell^2(\C^d))}^2. 
	\end{align*}
	In particular, if $\eta\in H^1_x$ then  Assumption \ref{ass:nonlinear}~\ref{ass:nonlinear_III} holds for $\nu = \infty$.
\end{remark}
\begin{remark}
    \label{rem:stoch_int}
    The most straightforward way to understand stochastic integrals with respect to $W$ satisfying Assumption \ref{ass:nonlinear}~\ref{ass:nonlinear_III} is to define them with respect to the underlying cylindrical Wiener process $\beta$ on $\ell^2$ given by
\[
(\beta,a)_{\ell^2} \,=\, \sum_{k\in \N} a_k\beta_k.
\]
Accordingly, we represent the space of Hilbert--Schmidt operators $L_2(\ell^2,H)$ for a separable Hilbert space $H$ as the sequence space $\ell^2(H)$ by identifying an operator  $\Gamma\colon \ell^2\to H$ with the sequence $(\Gamma e_k)_{k\in \N}$, where $e_k$ is the $k$-th unit vector. Then stochastic integrands 
\[
\Gamma\, \in\, L^0_{\mathcal{P}}(L^2(0,T;\ell^2(H)))
\]
are admissible to define the It\^o integral
\[
\int_0^\cdot \Gamma \,\d \beta \,\in\, L^0_\mathcal{P}(C([0,T]; H)),
\]
as an adapted local martingale.
Now for progressively strongly measurable
\[
\tilde{\Gamma} \colon \Omega\times [0,T] \to L(L^{\infty}_x\cap \dot{W}^{1,\nu}_x,H),
\]
we can define 
\[
\int_0^\cdot \tilde{\Gamma} \,\d W \,\coloneq\, \int_0^\cdot \tilde{\Gamma}\phi \,\d \beta \,\in \, L_\mathcal{P}^0(C([0,T]; H)),
\]
as soon as
\[
\sum_{k\in \N} \int_0^T \|\tilde{\Gamma} \phi_k \|_{H}^2\,\d t \,=\,\|\tilde{\Gamma}\phi \|_{L^2(0,T;\ell^2(H))}^2 \,<\,\infty,\qquad \text{$\P$-a.s.}
\]
\end{remark}

While parabolic (stochastic) PDE feature a smoothing effect of the linear part, the dispersive linear part of \eqref{eq:general_SNLS} leads to a more subtle gain of integrability due to the fact that different Fourier modes travel with different velocities. Respecting the scaling of the differential operators, cf. Remark \ref{rem:subcritical}, we can expect the solution to exist in $L^r(0,T;W_x^{1,q})$ for the following range of exponents, when started from $H^1_x$, see, e.g., \cite[Section 2.3]{cazenave} or \cite[Section 2.3]{tao_book}.

\begin{definition}[Stochastically admissible]
    \label{def:admissible}
    A pair of exponents $(r,q)$ with $r \in [2,\infty]$ and $q\in [2,\infty)$ is called \emph{stochastically admissible} if 
    \begin{equation}
        \label{eq:admissible}
        \frac{2}{r} + \frac{d}{q} = \frac{d}{2}.
    \end{equation}
\end{definition}
\emph{For the purpose of this manuscript, we use the terms admissible and  stochastically admissible synonymously.}

Regarding the formulation {stochastically admissible}, 
it should be noted that while it is customary to exclude the endpoint $(d,r,q)=(2,2,\infty)$ in the definition of admissible exponents in the literature on deterministic Schr\"odinger equations, cf.\ \cite[Definition 2.3.1]{cazenave}, we additionally exclude the triple $(d,r,q)=(1,4,\infty)$ satisfying \eqref{eq:admissible}.
The latter is related to the failure of the vector-valued Burkholder--Davis--Gundy inequality in $L^\infty$, which is vital to the proof of the stochastic Strichartz estimates stated in Theorem \ref{thm:stoch_Strichartz} below.
Thus, the following ranges for the parameters $r$ and $q$ can be realized as a component of a (stochastically) admissible pair:
\begin{equation}\label{eq:characterization_exponents}
    \begin{cases}
        r\in (4,\infty] \; \iff \; q\in [2,\infty), &d=1,\\
        r\in (2,\infty] \; \iff \; q\in [2,\infty) , & d=2,\\
        r\in [2,\infty] \; \iff\; q\in [2,\frac{2d}{d-2}], & d>2.
    \end{cases}
\end{equation}
According to \eqref{eq:admissible}, the lower boundary of the interval for $r$ corresponds to the upper boundary for the interval for $q$ and vice versa.

In the following, the parameter $p$ is the growth exponent on the nonlinearities from Assumption \ref{ass:nonlinear}. Since we impose the condition $p\in (1,1+\frac{4}{(d-2)_+})$, we have $p+1 \in  (2, \frac{2d}{(d-2)_+})$, so that
\begin{equation}\label{eq:Sobolev_embedding}
1 - \frac{d}{2} \,>\, - \frac{d}{p+1} \qquad \text{and thus}\qquad
H^1_x\, \hookrightarrow \,L^{p+1}_x,
\end{equation}
 by the Sobolev embedding theorem.  Inspecting \eqref{eq:characterization_exponents} we deduce that there exists $r\in (2,\infty)$ such that $(r,p+1)$ is admissible.
 We also recall the Schr\"odinger group
given by
\begin{equation}\label{eq:Schrodinger_sg}
S(t) \psi \,=\, \mathscr{F}^{-1}\bigl(\exp( -4\pi^2 i |\xi|^2 t) \mathscr{F}   \psi \bigr),\qquad t\in \R,
\end{equation}
for a Schwartz distribution $\psi \in \mathscr{S}'(\R^d)$, where we use the convention
\[
\mathscr{F} \psi(\xi ) \,=\, \int_{\R^d} \exp(-2\pi i x\cdot \xi)\psi (x)\,\d x.
\]
The latter is related to \eqref{eq:general_SNLS} since $u(t) = S(t)\psi$ solves the linear Cauchy problem $(\partial_t - i\Delta)u = 0$ with initial condition $u(0) = \psi$.
\begin{definition}[Maximal, unique local mild solution]\label{defi_local_mild_sol}Suppose that Assumption \ref{ass:nonlinear} holds and let $r$ be such that $(r,p+1)$ is admissible.
    \begin{enumerate}[label=(\roman*)]
        \item Let $\tau$ be a stopping time and $u\in C([0,\tau); H^1_x)\cap L_{\mathrm{loc}}^r([0,\tau);W_x^{1,p+1})$, $\mathbb{P}$-a.s., be progressively measurable. Then
        $(u,\tau)$ is called a \emph{local mild solution} to \eqref{eq:general_SNLS}, if, $\P$-a.s., the mild solution formula
        \begin{equation}\label{eq:mild_solution_SNLS}
		u(t) = S(t)u_0 + \int_0^t S(t-t')F(\cdot, u(t')) \d t' + \int_0^t S(t-t')G(u(t')) \d W(t'),
	\end{equation}
    holds for all $t\in [0,\tau)$.
    \item 
    A local mild solution $(u,\tau)$ is called \emph{maximal} and \emph{unique}, if for any other local mild solution $(v,\sigma)$ we have, $\P$-a.s.,\ $\sigma\le\tau$ and $u=v$ on $[0,\sigma)$.
    \end{enumerate}
\end{definition}
Assumption \ref{ass:nonlinear} ensures that all the integrals on the right-hand side of \eqref{eq:mild_solution_SNLS} exist $\P$-a.s.\ in $C([0,\tau);H^1_x)$, as can be seen from the proof of the main result of this section.
\begin{theorem}[Local well-posedness in $H^1$]\label{Thm:WP_H1}
	Suppose that Assumption \ref{ass:nonlinear} holds. Then there exists a maximal, unique  local mild solution $(u,\tau)$ to \eqref{eq:general_SNLS} in the sense of Definition \ref{defi_local_mild_sol}
	with $\tau>0$, $\P$-almost surely.
    The solution satisfies the following blow-up alternative:
	\begin{equation}\label{eq:blow_up}
		\PP[\Big]{\tau < \infty,\, \sup_{t\in [0, \tau)}\norm{u(t)}_{L^{p+1}_x} <\infty } \,=\, 0,
	\end{equation}
	and $u \in L_{\mathrm{loc}}^{\tilde{r}}([0,\tau ) ;W^{1,\tilde{q}}_x)$, $\mathbb{P}$-a.s., for every admissible $(\tilde{r},\tilde{q}
    )$.
\end{theorem}
By \eqref{eq:Sobolev_embedding}, the blow-up criterion \eqref{eq:blow_up} is stronger (i.e., easier to check) than one might expect from the fact that the initial value lies in $H_x^1$.

The rest of this section is devoted to the proof of Theorem \ref{Thm:WP_H1} and structured as follows: In Subsection \ref{Sec:Strichartz} we recall deterministic and stochastic Strichartz estimates, while we provide estimates on the nonlinearities $F^{(n)}$ and $G$ in Subsection \ref{Sec:nonlinearities}.  In Subsection \ref{Sec:truncation} we derive Lipschitz estimates on the truncation map in $L^{p+1}_x$ and  in Subsection \ref{Sec:WP_cutoff} well-posedness results for truncated versions  of \eqref{eq:general_SNLS}. Using that suitably truncated equations coincide with \eqref{eq:general_SNLS} on bounded subsets of $L^{p+1}_x$, we  deduce Theorem \ref{Thm:WP_H1} in Subsection \ref{Sec:Transf_to_orig_prob}.

\subsection{Deterministic and stochastic Strichartz estimates}
\label{Sec:Strichartz}
The Schr\"odinger group 
\eqref{eq:Schrodinger_sg} defines a group of unitary operators  on $ L^2_x $ with the 
skew-adjoint generator $i\Delta \colon
H_x^2 \to L_x^2$, see \cite[Section 2.1]{cazenave} for details. Since $S(t)$ commutes  with differentiation in space, it follows that
\begin{equation}\label{Eq112}
    \|S(\cdot )\psi  \|_{C([0,T]; H_x^{k})} \,=\, \|\psi \|_{H^k_x},
\end{equation}
for $k\in \{0,1\}$, $T\in (0,\infty)$ and any $\psi\in H^k_x$. 
The space on the left-hand side corresponds to the admissible pair $(\infty,2)$, the following generalizes \eqref{Eq112}  to all admissible pairs. 

\begin{theorem}[Homogeneous Strichartz estimates]
\label{thm:hom_Strichartz}
Let $(r,q)$ be an admissible pair. Then there exists a constant $\gamma<\infty$, such that 
\begin{align}\label{eq:homogeneous_Straichartz}
        \norm{S(\cdot)\psi}_{L^r(0,T;W^{k,q}_x)} \, &\leq \,\gamma \norm{\psi}_{H^k_x}, 
\end{align}
for all $k \in \{0,1\}$, $T \in (0,\infty)$ and $\psi \in H^k_x$.
\end{theorem}
For a proof we refer to \cite[Theorem 2.3.3 and Remark 2.3.8]{cazenave} and to \cite{keel_tao} for the endpoint case $(r,q) = (2,\frac{2d}{d-2})$, for  $d>2$.
Temporal convolutions against the Schr\"odinger group admit similar estimates:

\begin{theorem}[Inhomogeneous Strichartz estimates]
\label{thm:strichartz}Let $(r,q)$ and $(\tilde{r},\tilde{q})$ be admissible pairs. Then there exists a constant $\gamma<\infty$, such that 
\begin{equation*}
        \norm[\Big]{\int_0^{\cdot} S({\cdot}-t')h(t')\d t'}_{L^r(0,T;W^{k,q}_x)} \leq \gamma \norm{h}_{L^{\tilde{r}'}(0,T;W^{k,\tilde{q}'}_x)},
\end{equation*}
     for all $k \in \{0,1\}$, $T \in (0,\infty)$ and  $h \in L^{\tilde{r}'}(0,T;W^{k,\tilde{q}'}_x)$.
\end{theorem}
For a proof we  refer  again to \cite[Theorem 2.3.3 and Remark 2.3.8]{cazenave} and \cite{keel_tao}. We remark that in the case $(r,q) = (\infty,2)$ the $L^\infty(0,T;H_x^{k})$-norm can be replaced by $C([0,T];H_x^{k})$, as can be seen by approximating the right-hand side $h$ by smooth functions.

We also recall the stochastic Strichartz estimates addressing the stochastic convolution with $S(t)$, which were first derived in \cite[Section 3]{stoch_strichartz} to study stochastic Schr\"odinger equations on manifolds. They were subsequently generalized in \cite[Proposition 2]{hornung_SNLS} and we refer also to \cite[Appendix B]{gnann2024solitary} for a concise proof. To state them, we recall that the family $(\beta_k)_{k\in \N}$ of independent $\cF$-Brownian motions induces a  cylindrical Wiener process $\beta$ on $\ell^2$ as laid out in Remark \ref{rem:stoch_int}.

\begin{theorem}[Stochastic Strichartz estimates]\label{thm:stoch_Strichartz}
Let $(r,q)$ be an admissible pair. Then there exists a constant $\gamma<\infty$, such that
\begin{equation*}
      \norm[\Big]{\int_0^{\cdot} S(\cdot - t')h(t') \d \beta({t'})}_{L^{{ \zeta}}_{\mathcal P}(L^r(0,T;W^{k,q}_x))} \,\leq\, \gamma  \sqrt{ \zeta} \norm{h}_{ L^{{ \zeta } }_{\mathcal{P}}(L^2(0,T;\ell^2(H^k_x)))},
\end{equation*}
for all $k \in \{0,1\}$, $T \in (0,\infty)$,  $\zeta\in [2,\infty)$
 and  $h \in L^{\zeta}_{\mathcal{P}}(L^2(0,T;\ell^2(H^k_x)))$.
\end{theorem}
Also here we can use the $C([0,T]; H^k_x)$-norm  in the case $(r,q)=(\infty,2)$.
For later purposes, we rephrase the stochastic Strichartz estimates for stochastic convolutions with respect to a noise $W$ as in \eqref{Eq15}.

\begin{corollary}\label{lemma:stochStrich}Let $(r,q)$ be an admissible pair and  $\nu$ as in \eqref{eqn:nu_exp}. Then there exists a constant $\gamma<\infty$, such that
\begin{equation*}
	\norm[\Big]{\int_0^{\cdot} S(\cdot - t')g(t') \d W({t'})}_{L^{{ \zeta}}_{\mathcal P}(L^r(0,T;W^{k,q}_x))} \leq \gamma { \sqrt{ \zeta}}\bigl( \norm{\phi}_{L^{\infty}(\R^d;\ell^2)} +\delta_{k=1} \|\nabla \phi\|_{ L^{\nu}(\R^d;\ell^2(\C^d))}
	\bigr)
	\norm{g}_{ L^{  \zeta}_{\mathcal{P}}(L^2(0,T;H^k_x))} ,
\end{equation*}
for all $k \in \{0,1\}$, $T \in (0,\infty)$, $\zeta\in [2,\infty)$,  $g \in L^{\zeta}_{\mathcal{P}}(L^2(0,T;H^k_x))$ and noises $W$ with expansion \eqref{Eq15} for $(\phi_k)_{k\in \N}\in L^\infty(\R^d;\ell^2)$ satisfying additionally $(\nabla \phi_k)_{k\in \N}\in L^{\nu}(\R^d;\ell^2(\C^d))$ if $k=1$.
\end{corollary}
\begin{proof}
Following Remark \ref{rem:stoch_int} we have 
\begin{align*}
	\int_0^t S(t-t')g(t')\, \d W(t') \,=\, \int_0^t 
	S(t-t') h(t')
	\,\d\beta ,
\end{align*}
where $h_k(t) = g(t)\phi_k$. Regarding the new integrand $h$, we observe that
\begin{align*}
	\|h\|_{\ell^2(L^2_x)} \,=\, 
	\|h\|_{L^2(\R^d;\ell^2)} \,\le\, &
	\|g \|_{L^2_x}\| \phi \|_{L^\infty(\R^d; \ell^2)} ,\\
	\|\partial_{x_l} h\|_{\ell^2(L^2_x)} \,\le\, 
	\|(\partial_{x_l}g) \phi \|_{L^2(\R^d;\ell^2)} \,+\, 
	\| g (\partial_{x_l} \phi) \|_{L^2(\R^d;\ell^2)}\,\le \, &
	\|\partial_{x_l}g\|_{L^2_x} \| \phi \|_{L^\infty(\R^d;\ell^2)} \,+\, 
	\|g\|_{L^{\mu}_x} \| \partial_{x_l}\phi \|_{L^\nu(\R^d;\ell^2)} ,
\end{align*}
$\mathbb{P}\otimes \d t$-almost everywhere
due to H\"older's inequality with $1/2 = 1/\mu + 1/\nu$, for $l=1,\dots , d$. For $d\ge 3$, we have the Sobolev embedding $H^1_x\hookrightarrow L^2_x\cap L^{2d/(d-2)}_x \hookrightarrow L^\mu_x$, where the latter embedding follows since $\nu\ge d$. For $d\le 2$ we have $H^1_x \hookrightarrow L^\mu_x$ as well, using non-critical Sobolev embeddings and the assumption \eqref{eqn:nu_exp}.
All in all, we obtain
the claim from Theorem \ref{thm:stoch_Strichartz} after enlarging $\gamma$.
\end{proof}
Similarly, one can obtain Strichartz estimates for deterministic convolutions with the Schr\"odinger semigroup when composed with a multiplication operator.
\begin{corollary}\label{cor:mult_strichartz}Let $(r,q)$ and $(\tilde{r},\tilde{q})$ be admissible pairs and $\nu$ as in \eqref{eqn:nu_exp}. Then there exists a constant $\gamma<\infty$, such that 
	\begin{equation*}
		\norm[\Big]{\int_0^{\cdot} S({\cdot}-t')\bigl(\Psi f(t')\bigr)\d t'}_{L^r(0,T;W^{k,q}_x)} \leq \gamma \bigl( \norm{\Psi}_{L^{\infty}_x} +\delta_{k=1} \|\nabla \Psi\|_{ L^{\nu}(\R^d;\C^d)}
		\bigr) \norm{f}_{L^{\tilde{r}'}(0,T;W^{k,\tilde{q}'}_x)},
	\end{equation*}
	for all $k \in \{0,1\}$, $T \in (0,\infty)$,  $f \in L^{\tilde{r}'}(0,T;W^{k,\tilde{q}'}_x)$ and functions  $\Psi \in L_x^\infty$ satisfying additionally $\nabla \Psi\in L^{\nu}(\R^d;\C^d)$ if $k=1$.
\end{corollary}
\begin{proof}We define $h(t) = \Psi f(t)$ and obtain 
	\begin{align*}
		\|h\|_{L_x^{\tilde{q}'}}
		 \,\le&{\, }
		\|f \|_{L^{\tilde{q}'}_x}\| \Psi \|_{L_x^\infty} ,\\
		\|\partial_{x_l} h\|_{L_x^{\tilde{q}'}} \,\le\, 
		\|(\partial_{x_l}f) \Psi \|_{L_x^{\tilde{q}'}} \,+\, 
		\|f (\partial_{x_l} \Psi) \|_{L_x^{\tilde{q}'}}\,\le &{\, }
		\|\partial_{x_l}f\|_{L^{\tilde{q}'}_x} \| \Psi \|_{L^\infty_x} \,+\, 
		\|f\|_{L^{\mu}_x} \| \partial_{x_l}\Psi \|_{L_x^\nu} ,
	\end{align*}
where here $1/\tilde{q}' = 1/\mu +1/\nu$. For $d\ge 3$, since $\tilde{q}\ge 2$ by admissibility, we have that $W^{1,\tilde{q}'}_x\hookrightarrow L^{\tilde{q}'}_x\cap L^{d\tilde{q}'/(d-\tilde{q}')}_x \hookrightarrow L^\mu_x$ and the latter is a consequence of $\nu\ge d$. For $d=2$ and $\tilde{q}>2$ the same applies, while for $d=2$ and $\tilde{q}=2$ or $d=1$ we use non-critical Sobolev embeddings based on \eqref{eqn:nu_exp} to see that  $W^{1,\tilde{q}'}_x \hookrightarrow L^\mu_x$. The claim follows by Theorem \ref{thm:strichartz} up to enlarging $\gamma$. 
\end{proof}
\subsection{ Estimates on the nonlinearities}\label{Sec:nonlinearities}
In this subsection we consider nonlinearities ${F}^{(n)}$ and $G$ as in 
Assumption \ref{ass:nonlinear}~\ref{ass:nonlinear_II}--\ref{ass:nonlinear_G} and take $p\in (1,1+\frac{4}{(d-2)_+})$ such that the latter  holds. 
As a consequence, there exist functions ${F_1^{(n)}, F_2^{(n)},G_1,G_2 \in C^1(\R^2;\C)}$ with $F_1^{(n)}(0) = F_2^{(n)}(0) = G_1(0) = G_2(0)=0$ such that
\begin{align}\label{Eq29}
    F^{(n)} = F_1^{(n)} + F_2^{(n)}, \qquad G = G_1 + G_2,
\end{align}
and\noeqref{eq:F1pwlip,eq:G1pwlip,eq:F2pwlip,eq:G2pwlip}
\begin{subequations}
\label{eq:FGpwlip}
\begin{align}
    \label{eq:F1pwlip}
    \lvert F_1^{(n)}(x) - F_1^{(n)}(y) \rvert &\leq C \lvert x - y \rvert, \\
    \label{eq:G1pwlip}
    \lvert G_1(x) - G_1(y) \rvert &\leq C \lvert x - y \rvert, \\
    \label{eq:F2pwlip}
    \lvert F_2^{(n)}(x) - F_2^{(n)}(y) \rvert &\leq C (\lvert x \rvert^{p-1} + \lvert y \rvert^{p-1}) \lvert x - y \rvert, \\
    \label{eq:G2pwlip}
    \lvert G_2(x) - G_2(y) \rvert &\leq C (\lvert x \rvert^{\frac{p-1}{2}} + \lvert y \rvert^{\frac{p-1}{2}})\lvert x - y \rvert,
\end{align}
\end{subequations}
for some new constant $C<\infty$, all $x,y\in \R^2$ and $n\in \{1,\dots ,N\}$. Indeed, this can be achieved  by setting $F_1^{(n)}(x) = \eta(|x|) F^{(n)}(x)$ and $G_1(x) = \eta(|x|) G(x)$, where $\eta \colon [0,\infty) \to \R$ is smooth, decreasing and satisfies $\eta(s) = 1$ for $s\le 1$ and $\eta(s) = 0$ for $s\ge 2$. We derive estimates for $(F_1^{(1)},\dots, F^{(N)}_1,G_1)$ and $(F_2^{(1)},\dots, F^{(N)}_2,G_2)$ separately.

\begin{lemma}[Lipschitz estimates on $(F_1^{(n)},G_1)$]
\label{lem:F1G1est} There exists  $ M<\infty $ such that the estimates 
\begin{subequations}
\label{eq:F1G1est}
\begin{align}\label{eq:bddness_F1}
    \lVert F_1^{(n)}(u) \rVert_{H^1_x} &\leq M \lVert u \rVert_{H^1_x}, \\
    \label{eq:bddness_G1}
    \lVert G_1(u) \rVert_{H^1_x} &\leq M \lVert u \rVert_{H^1_x}, \\\label{eq:Lipschtiz_est_F1}
    \lVert F_1^{(n)}(u) - F_1^{(n)}(v) \rVert_{L^2_x} &\leq M \lVert u - v \rVert_{L^2_x}, \\
    \label{eq:Lipschtiz_est_G1}
    \lVert G_1(u) - G_1(v) \rVert_{L^2_x} &\leq M \lVert u - v \rVert_{L^2_x},
\end{align}
\end{subequations}
hold for all $u, v \in H^1_x$ and $n\in \{1,\dots, N\}$.
\end{lemma}
\begin{proof}
It suffices to prove the assertions regarding $F_1^{(1)}$, for which we write $F_1=F_1^{(1)}$ in this proof.  By $F_1(0) = 0$ and \eqref{eq:F1pwlip} we have
\begin{align*}
\|F_1(u)\|_{L^2_x} \,=\, 
\|F_1(u) -F_1(0)\|_{L^2_x} \,\le\,C \|u \|_{L^2_x}.
\end{align*}
Moreover, as $\partial_{x_l} (F_1(u)) = F_1'(u) \partial_{x_l}u$ by \cite[Lemma 7.5]{gilbarg_trudinger}, we obtain also that $\| \partial_{x_l} F_1(u) \|_{L^2_x(\R^d)} \le C \|\partial_{x_l} u \|_{L_x^2}$ for $l=1,\dots, d$ completing the proof of \eqref{eq:bddness_F1}. The claim \eqref{eq:Lipschtiz_est_F1} follows immediately from \eqref{eq:F1pwlip}.
\end{proof}
\begin{lemma}[Local Lipschitz estimates on $(F_2^{(n)},G_2)$]\label{lemma:bounds_FG}
There exists  $ M<\infty $ such that the estimates
\label{lem:F2est}
\begin{subequations}
\label{eq:F2G2est}
\begin{align} \label{eq:F2_bound}
    \lVert F_2^{(n)}(u) \rVert_{W^{1,1+1/p}_x} &\leq M \lVert u \rVert_{L^{p+1}_x}^{p-1} \lVert u \rVert_{W^{1,p+1}_x}, \\ \label{eq:G2_bound}
    \lVert G_2(u) \rVert_{H^1_x} &\leq M \lVert u \rVert_{L^{p+1}_x}^{{(p-1)}/{2}} \lVert u \rVert_{W^{1,p+1}_x}, \\ \label{eq:F2_lip}
    \lVert F_2^{(n)}(u) - F_2^{(n)}(v) \rVert_{L^{1+{1}/{p}}_x} &\leq M \bigl(\lVert u \rVert_{L^{p+1}_x}^{p-1} + \lVert v \rVert_{L^{p+1}_x}^{p-1}\bigr) \lVert u - v \rVert_{L^{p+1}_x}, \\
    \label{eq:G2_lip}
    \lVert G_2(u) - G_2(v) \rVert_{L^2_x} &\leq M \bigl(\lVert u \rVert_{L^{p+1}_x}^{{(p-1)}/{2}} + \lVert v \rVert_{L^{p+1}_x}^{{(p-1)}/{2}}\bigr) \lVert u - v \rVert_{L^{p+1}_x},
\end{align}
\end{subequations}
hold for all $u, v \in W^{1,p+1}_x$  and $n\in\{1,\dots, N\}$.
\end{lemma}
\begin{proof}
We start by proving the assertions concerning $F_2^{(n)}$ for which it suffices again to consider only $F_2\coloneq F_2^{(1)}$. Firstly, we observe that 
\begin{align}\begin{split}\label{Eq21}&
    \norm{F_2(u) -F_2(v)}_{L_x^{1+{1}/{p}}} \,{\le}\,C\norm[big]{(|u|^{p-1} +|v|^{p-1})|u-v| }_{L_x^{1+{1}/{p}}}
   \,\le \, C \bigl(\|u\|_{L_x^{p+1}}^{p-1} +  \|v\|_{L_x^{p+1}}^{p-1}\bigr)\|u-v\|_{L_x^{p+1}},
    \end{split}
\end{align}
where we used \eqref{eq:F2pwlip} and H\"older's inequality with exponents
\begin{equation*}
    \frac{p}{p+1}\,=\, \frac{p-1}{p+1}\,+\, \frac{1}{p+1}.
\end{equation*}
This finishes the proof of \eqref{eq:F2_lip} and by setting $v=0$, we also recover the bound
\begin{equation}\label{Eq19}
\|F_2(u)\|_{L_x^{1+{1}/{p}}} \,\le\, 
C\|u\|_{L_x^{p+1}}^{p-1} \|u\|_{L_x^{p+1}}.
\end{equation}
To complete also the proof of \eqref{eq:F2_bound}, since unlike in Lemma \ref{lem:F1G1est} we compose with non-Lipschitz functions, we introduce the 
difference quotients
\begin{equation*}
    D^h_l u(x) \,=\, \frac{u(x+h e_l) - u(x)}{h}
    ,\qquad x\in \R^d,
\end{equation*}
which obey the bound
\[
\|D^h_l u\|_{L_x^{p+1}} \,\le\, \|\partial_{x_l} u \|_{L_x^{p+1}},
\]
by the fundamental theorem of calculus and Minkowski's inequality. Hence, we can estimate
\begin{align}&\label{Eq23}
    \|D^h_l F_2(u) \|_{L_x^{1+{1}/{p}}}  \, 
    {\le}\,C \norm[big]{\bigl( |u(\cdot + he_l )|^{p-1} + |u|^{p-1} \bigr) D^h_l u}_{L_x^{1 + 1/p}}
    \,\le\, 2C \|u\|_{L^{p+1}_x}^{p-1}\|\partial_{x_l} u \|_{L_x^{p+1}},
\end{align}
by another application of \eqref{eq:F2pwlip} and  H\"older's inequality. A weak convergence argument yields that the same bound holds for $\partial_{x_l} F_2(u)$, see \cite[Lemma 7.24]{gilbarg_trudinger}, resulting together with \eqref{Eq19} in \eqref{eq:F2_bound}.

The estimates \eqref{eq:G2_lip} and \eqref{eq:G2_bound} follow similarly to \eqref{Eq21} and \eqref{Eq23} by employing \eqref{eq:G2pwlip} and 
H\"older's inequality with 
\[
\frac{1}{2} \,=\, \frac{p-1}{2(p+1)}\,+\, \frac{1}{p+1}. \qedhere
\]
\end{proof}

\subsection{Estimates on the truncation mapping}\label{Sec:truncation}To turn the estimates from the previous subsection into global Lipschitz estimates, we introduce suitable truncation maps by defining functions 
\begin{equation*}
\begin{aligned}
    \theta_R\colon [0,\infty) \to [0,1], \, s \mapsto  \begin{cases}
        1, \quad &s \leq R, \\
        {R}/{s}, \quad &s> R,
    \end{cases}
\end{aligned}
\end{equation*}
for $R \in (0,\infty)$. 
This allows us to define the mappings 
\begin{equation*}
\begin{aligned}
    \Phi_R \colon L_x^{p+1} &\to L_x^{p+1} ,\,
    u \mapsto \theta_R(\norm{u}_{L^{p+1}_x})u,
\end{aligned}
\end{equation*}
which satisfy the following properties.

\begin{lemma}[Lipschitz truncation]
    \label{lem:truncation}
    Let $R \in (0,\infty)$, then the estimates 
    \begin{subequations}
    \label{eq:PhiRest}
    \begin{align}
        \label{eq:PhiRtrunc}
        \norm{ \Phi_R(u) }_{L^{p+1}_x} &\leq R, \\
        \label{eq:PhiRlip}
        \norm{ \Phi_R(u) - \Phi_R(v) }_{L^{p+1}_x} &\leq 2\norm{ u - v }_{L^{p+1}_x},
    \end{align}
hold
    for all  $u,v \in L^{p+1}_x$.
    \end{subequations}
\end{lemma}

\begin{proof}
    We write $\alpha = \norm{ u }_{L^{p+1}_x}$.
    By definition of $\Phi_R$, we get
    \begin{equation*}
        \norm{ \Phi_R(u) }_{L^{p+1}_x} \, = \, \norm{ \theta_R(\alpha) u}_{L^{p+1}_x} \, =\, \alpha \theta_R(\alpha) \, =\,  \alpha \wedge R \, \leq \, R,
    \end{equation*}
    which shows \eqref{eq:PhiRtrunc}.
    If we additionally write $\beta = \norm{ v }_{L^{p+1}_x}$, we can estimate
    \begin{align*}
        \norm{ \Phi_R(u) - \Phi_R(v) }_{L^{p+1}_x} 
        &\,=\, \norm{ \theta_R(\alpha)u - \theta_R(\beta)v }_{L^{p+1}_x} \\
        &\,\leq\, \norm{ \theta_R(\alpha)u - \theta_R(\beta)u }_{L^{p+1}_x} + \norm{ \theta_R(\beta)u - \theta_R(\beta)v}_{L^{p+1}_x} \\
        &\,\leq\,  \alpha (\theta_R(\alpha) - \theta_R(\beta)) + \norm{ u - v}_{L^{p+1}_x},
    \end{align*}
    where we assume without loss of generality that $\alpha\leq \beta$. Then we proceed to bound
    \[
    \alpha (\theta_R(\alpha) - \theta_R(\beta)) \,=\, \alpha \wedge R  - \beta\wedge R + (\beta-\alpha) \theta_R(\beta)\,\le\, \beta -\alpha,
    \]
    so that an application of the reverse
    triangle inequality results in \eqref{eq:PhiRlip}.
\end{proof}

\subsection{Global in time well-posedness for the truncated equation}\label{Sec:WP_cutoff}
In this section, we consider the truncated version
\begin{equation*}
\begin{cases}
    \d  u \,=\, \Bigl[i \Delta u + \sum_{n=1}^N \Psi_n \bigl(F_1^{(n)}(u) \,+\,  {F}_2^{(n)}(\Phi_R (u))\bigr)\Bigr] \,\d t\, +\, \bigl[ G_1(u) + G_2( \Phi_R(u))\bigr] \, \d W(t),\\
    u(0)= u_0,
	\end{cases}
\end{equation*}
of \eqref{eq:general_SNLS} based on the decompositions \eqref{eqn_f_tilda_equals} and  \eqref{Eq29}.
The central observation is that as a consequence of the results from Subsections \ref{Sec:nonlinearities}--\ref{Sec:truncation} the modified nonlinearities $F_1^{(n)} + F_2^{(n)} \circ\Phi_R$ and
$ G_1+G_2\circ \Phi_R$ are globally Lipschitz for fixed $R>0$. Then the deterministic and stochastic Strichartz estimates from Subsection \ref{Sec:Strichartz} allow to set up a fixed point argument in the following space. 
\begin{lemma}\label{lemma:metric_space}
Let  $L,T \in (0,\infty)$. Then the set 
\begin{align}\label{eqn_AA}
    \mathscr{M}_{L,T}\,=\, \Bigl\{& 
    u \in L^2_{\mathcal{P}}(C([0,T]; L^2_x) \cap L^r(0,T; L_x^{p+1}))\,\Big|
    \, 
    \mathbb E^{1/2}\Bigl[ \|u\|_{L^\infty(0,T;H^1_x)\cap L^r(0,T;W_x^{1,p+1})}^2\Bigr] \le L
    \Bigr\},
\end{align}
equipped with the metric
\begin{equation*}
    d_{\mathscr{M}_{L,T}}(u,v)
    \,\coloneq\, \|u-v\|_{L^2_{\mathcal{P}}({C([0,T]; L^2_x)} \cap L^r(0,T;L^{p+1}_x))}
\end{equation*}
is complete.
\end{lemma}
\begin{proof}
    Let $(u_k)_{k\in \N}$ be Cauchy in $\mathscr{M}_{L,T}$. Then 
    \begin{equation}\label{eqn_conv_in_lemma}
        u_k \,\to\, u , \qquad \text{in } L^2_{\mathcal{P}}({C([0,T]; L^2_x)}\cap L^r(0,T;L^{p+1}_x)),
    \end{equation}
    by completeness of the latter space. 
    After passing to a subsequence along which the convergence \eqref{eqn_conv_in_lemma} holds additionally $\mathbb P$-a.s., we deduce that also
    \[
     \|u\|_{L^\infty(0,T;H^1_x)\cap L^r(0,T;W_x^{1,p+1})}\,
    \le\,
     \liminf_{k\to\infty} \| u_k \|_{L^\infty(0,T;H^1_x)\cap L^r(0,T;W_x^{1,p+1})}
    ,
    \]
    $\P$-a.s., by weak lower-semicontinuity of this norm with respect to ${C([0,T]; L^2_x)}$-convergence. Then, an application of Fatou's lemma yields that 
    \begin{align}
        \mathbb E^{1/2} \Bigl[ \|u\|_{L^\infty(0,T;H^1_x)\cap L^r(0,T;W_x^{1,p+1})}^2 \Bigr] \,\le\, \liminf_{k\to\infty} 
         \mathbb E^{1/2} \Bigl[ \|u_k\|_{L^\infty(0,T;H^1_x)\cap L^r(0,T;W_x^{1,p+1})}^2 \Bigr] \,\le\, L,
    \end{align}
    finishing the proof.  
\end{proof}
\begin{remark}
For subtle reasons connected to the non-separability of $L^\infty(0,T;H^1_x)$, the second moment in \eqref{eqn_AA} is written explicitly. Indeed, expressing it as a $L^2_\Omega$-norm instead would suggest strong measurability in  $L^\infty(0,T;H^1_x)$, which we refrain from in the definition of $\mathscr{M}_{L,T}$. Also in the following calculations, we make this distinction explicit, cf.\ \eqref{Eq266}, \eqref{eq:kdsnfdsjf} and \eqref{Eq33} below. For the obtained fixed point however, we do recover said strong measurability and even continuity in $H^1_x$ a posteriori, see Step 4 of  the following proof.
\end{remark}

We  recall the exponent $p$ from Assumption \ref{ass:nonlinear} and that $(r,p+1)$ is a Strichartz pair for some $r\in (2,\infty)$ by the comments below \eqref{eq:Sobolev_embedding}.
\begin{proposition}[Truncated equation---existence]\label{prop:trunc_ex} Suppose that Assumption \ref{ass:nonlinear} holds and let $R,T_0 \in (0,\infty)$. If additionally $u_0 \in L^2_{\Omega}(H^1_x)$ then there exists $u \in L^2_{\mathcal{P}}(C([0,T_0];H^1_x) \cap L^r(0,T_0;W^{1,p+1}_x))$ such that, $\P$-a.s., the mild solution formula
\begin{equation}
    \label{eq:nlsmildtruncated}
    \begin{aligned}
        u(t) \,=\, S(t)u_0\, &+\, {\sum_{n=1}^N }\int_0^t S(t-t')\Bigl[ \Psi_n\bigl({F}_1^{(n)}( u(t'))\, +\,  {F}_2^{(n)}(\Phi_R(u(t')))\bigr)\Bigr]\,\d t' \\
        &+ \,\int_0^t S(t-t')\bigl[G_1(u(t')) + G_2(\Phi_R(u(t')))\bigr] \, \d W(t'),
    \end{aligned}
\end{equation} holds for $t \in [0,T_0]$.
Moreover, we have $u\in L^2_{\mathcal{P}}(L^{\tilde{r}}(0,T_0;W^{1,\tilde{q}}_x))$ for every admissible $(\tilde{r},\tilde{q})$.
\end{proposition}

\begin{proof}
At the heart of this proof is to show that there exists $T\in (0,\infty)$ independent of $u_0$ and $L\in (0,\infty)$, such that 
    \begin{align}\begin{split}\label{eq:Xi}
        \Xi_{u_0} \colon 
         &(\mathscr{M}_{L,T},d_{\mathscr{M}_{L,T}}) \to (\mathscr{M}_{L,T},d_{\mathscr{M}_{L,T}}),
         \\ & u  \mapsto
         S(\cdot)u_0 \,+\, \textstyle{\sum_{n=1}^N}\bigl(P_n F_1^{(n)}(u) \,+\,  P_n F_2^{(n)}(\Phi_R(u))\bigr) \,+\, QG_1(u) \,+\,  Q G_2(\Phi_R(u)),
    \end{split}\end{align}
is a Lipschitz contraction, where we use the shorthand notations
    \begin{align*}
        P_n f (t) \,&=\,
        \int_0^{t} S({t}-t') \bigl(\Psi_n f(t')\bigr)\,\d t',
        \\
        Q g(t) \,&=\, 
        \int_0^{t} S({t}-t')g(t')\,\d W({t'}),
    \end{align*}
whenever these are well-defined.
We divide it into several steps.

\emph{Step 1 (Space-time bounds on truncated nonlinearities).} We let $n\in \{1,\dots, N\}$ and assume that 
\begin{equation}\label{Eq79}
	u,v\in \mathscr{M}_{L,T},
\end{equation} for $L,T\in (0,\infty)$, 
and let $M<\infty$ be such that the assertions of Lemmas \ref{lem:F1G1est}--\ref{lem:F2est} hold. Then, we notice that after integrating \eqref{eq:bddness_F1} and \eqref{eq:Lipschtiz_est_F1}, we have\noeqref{Eq25a,Eq25b}
\begin{subequations}\label{Eq25}
\begin{align} \label{Eq25a}
    \lVert F_1^{(n)}(u) \rVert_{L^1(0,T;H^1_x)} & \,\leq\,  M T \lVert u \rVert_{L^\infty(0,T;H^1_x)},\\
\label{Eq25b} \lVert F_1^{(n)}(u) - F_1^{(n)}(v) \rVert_{L^1(0,T;L^2_x)} &\,\leq\, M T \lVert u - v \rVert_{C([0,T]; L^2_x)},
\end{align}
\end{subequations}
$\mathbb{P}$-almost surely.
Similarly, after evaluating the $L^2$-norm in time of \eqref{eq:bddness_G1} and \eqref{eq:Lipschtiz_est_G1} and using H\"older's inequality, we find\noeqref{Eq26a,Eq26b}
\begin{subequations}
    \label{Eq26}
\begin{align}\label{Eq26a}
    \lVert G_1(u) \rVert_{L^2(0,T;H^1_x)} &\,\leq\, M T^{\frac{1}{2}} \lVert u \rVert_{L^\infty(0,T;H^1_x)},
    \\\label{Eq26b}
    \lVert G_1(u) - G_1(v) \rVert_{L^2(0,T;L^2_x)} &\,\leq\, M T^{\frac{1}{2}} \lVert u - v \rVert_{C([0,T]; L^2_x)},
\end{align}
\end{subequations}
$\mathbb{P}$-almost surely, as well.

Next, by the properties of the truncation \eqref{eq:PhiRest} and the estimates \eqref{eq:F2G2est} on $F_2^{(n)}$ and $G_2$, we deduce that
\begin{align*}
    \lVert F_2^{(n)}(\Phi_R(u)) \rVert_{W^{1,1+{1}/{p}}_x} &\,\leq\,  M R^{p-1} \lVert u \rVert_{W^{1,p+1}_x}, \\
    \lVert G_2(\Phi_R(u)) \rVert_{H^1_x} &\,\leq\,  M R^{{(p-1)}/{2}} \lVert u \rVert_{W^{1,p+1}_x}, \\
    \lVert F_2^{(n)}(\Phi_R(u)) - F_2^{(n)}(\Phi_R(v)) \rVert_{L^{1+{1}/{p}}_x} &\,\leq\, 4 M R^{p-1}\lVert u - v \rVert_{L^{p+1}_x}, \\
    \lVert G_2(\Phi_R(u)) - G_2(\Phi_R(v)) \rVert_{L^2_x} &\,\leq\, 4 M R^{{(p-1)}/{2}} \lVert u - v \rVert_{L^{p+1}_x},
\end{align*}
$\mathbb{P} \otimes \d t$-almost everywhere.
To proceed, we define $\alpha = {1}/{r'}-{1}/{r}$ and $\beta = {1}/{2}-{1}/{r}$, which lie both in $(0,1)$ since $r\in (2,\infty)$. Evaluating the $L^{r'}$-norm and $L^2$-norm, respectively,  in time together with an application of H\"older's inequality yields, $\P$-a.s.,\noeqref{Eq27a,Eq27b,Eq27c,Eq27d}
\begin{subequations}\label{Eq27}
\begin{align}
\label{Eq27a}
    \lVert F_2^{(n)}(\Phi_R(u)) \rVert_{L^{r'}(0,T;W^{1,1+{1}/{p}}_x)} &\,\leq\, M R^{p-1} T^{\alpha} \lVert u \rVert_{L^r(0,T;W^{1,p+1}_x)}, \\
\label{Eq27b}
    \lVert G_2(\Phi_R(u)) \rVert_{L^2(0,T;H^1_x)} &\,\leq\,  M R^{{(p-1)}/{2}} T^{\beta} \lVert u \rVert_{L^{r}(0,T;W^{1,p+1}_x)}, \\
\label{Eq27c}
    \lVert F_2^{(n)}(\Phi_R(u)) - F_2^{(n)}(\Phi_R(v)) \rVert_{L^{r'}(0,T;L^{1+{1}/{p}}_x)} &\,\leq\, 4 M R^{p-1} T^{\alpha} \lVert u - v \rVert_{L^r(0,T;L^{p+1}_x)}, \\
\label{Eq27d}
    \lVert G_2(\Phi_R(u)) - G_2(\Phi_R(v)) \rVert_{L^2(0,T;L^2_x)} &\,\leq\, 4 M R^{{(p-1)}/{2}} T^{\beta} \lVert u - v \rVert_{L^r(0,T;L^{p+1}_x)}.
\end{align}
\end{subequations}

\emph{Step 2 (Estimates on the terms on the right-hand side of \eqref{eq:Xi}).} For the rest of the proof we fix $\gamma<\infty$ such that the assertions of Theorem \ref{thm:hom_Strichartz} and Corollaries \ref{lemma:stochStrich}--\ref{cor:mult_strichartz} hold for the admissible pairs $(\infty,2)$ and $(r,p+1)$. 
Then, from the homogeneous Strichartz estimates, Theorem \ref{thm:hom_Strichartz}, we see that
\begin{equation}\label{eq:growth1}
    \|S(\cdot)u_0\|_{Y} \,\le\,2 \gamma \|u_0\|_{H^1_x},
\end{equation}
$\P$-a.s.,
where we employ the notation
\begin{align*}
    X\,&=\, C([0,T];L^2_x)\cap L^r(0,T;L_x^{p+1}),\\
    Y \,&=\, C([0,T];H^1_x)\cap L^r(0,T;W_x^{1,p+1}).
\end{align*}
By combining the inhomogeneous Strichartz estimates from Corollary \ref{cor:mult_strichartz} 
with the bounds \eqref{Eq25} we deduce that $\P$-a.s.\
\begin{subequations}
    \begin{align} \label{eq:growth2}
          \lVert P_n F_1^{(n)}(u) \rVert_{Y} &\,\leq\, 2\gamma C_{\Psi_{n}}  M T  \lVert u \rVert_{L^\infty(0,T;H^1_x)},\\
\lVert P_n F_1^{(n)}(u) - P_n F_1^{(n)}(v) \rVert_{X} &\label{eq:lip1}\,\leq\, 2\gamma \tilde{C}_{\Psi_{n}} M T \lVert u - v \rVert_{C([0,T]; L^2_x)},
    \end{align}
\end{subequations}
where we set 
\begin{equation}\label{constants_dep_on_psi}
C_{\Psi_{n}}  \,=\, 
\|\Psi_n\|_{L^\infty_x} \,+\, \|\nabla\Psi_n\|_{L^\nu(\R^d;\C^d)},\qquad 
\tilde{C}_{\Psi_{n}}  \,=\, \|\Psi_n\|_{L^\infty_x},\qquad 
n \in \{1,\dots, N\}.
\end{equation}
Similarly, combining the stochastic Strichartz estimates from Corollary \ref{lemma:stochStrich}  with  \eqref{Eq26} we obtain 
\begin{subequations}
    \label{Eq266}
\begin{align} \label{eq:growth3}
    \lVert QG_1(u) \rVert_{{L^2_{\mathcal P}(Y)}} &\,\leq\, 2\gamma C_\phi M T^{\frac{1}{2}} \mathbb E^{1/2}\bigl[\lVert u \rVert_{L^\infty(0,T;H^1_x)}^2\bigr] , 
    \\\label{eq:lip2}
    \lVert QG_1(u) - QG_1(v) \rVert_{{L^2_{\mathcal P}(X)}} &\,\leq\, 2\gamma \tilde{C}_\phi M T^{\frac{1}{2}}  \lVert u - v \rVert_{L_{\mathcal{P}}^2(C([0,T]; L^2_x))},
\end{align}
\end{subequations}
with 
\begin{align}\label{constants_dep_on_phi}
C_\phi \,=\, \|\phi\|_{L^\infty(\R^d;\ell^2)} \,+\, 
\|\nabla \phi\|_{L^\nu(\R^d;\ell^2(\C^d))}
,\qquad \tilde{C}_\phi\,=\, \|\phi\|_{L^\infty(\R^d;\ell^2)}.
\end{align}
We proceed analogously using the estimates \eqref{Eq27} to obtain
\begin{subequations}\label{eq:kdsnfdsjf}
\begin{align} \label{eq:growth4}
    \lVert P_n F_2^{(n)}(\Phi_R(u)) \rVert_{Y} &\,\leq\, 2\gamma C_{\Psi_n} M R^{p-1} T^{\alpha} \lVert u \rVert_{L^r(0,T;W^{1,p+1}_x)}, \\
    \label{eq:growth5}
    \lVert Q G_2(\Phi_R(u)) \rVert_{L_{\mathcal P}^2(Y)} &\,\leq\, 2\gamma C_\phi  M R^{\frac{p-1}{2}} T^{\beta} \mathbb E^{1/2} \bigl[ \lVert u \rVert_{L^{r}(0,T;W^{1,p+1}_x)}^2\bigr] ,
    \\\label{eq:lip3}
    \lVert P_nF_2^{(n)}(\Phi_R(u)) - P_nF_2^{(n)}(\Phi_R(v)) \rVert_{X} &\,\leq\, 8\gamma \tilde{C}_{\Psi_n}M R^{p-1} T^{\alpha} \lVert u - v \rVert_{L^r(0,T;L^{p+1}_x)}, \\
    \label{eq:lip4}
    \lVert Q G_2(\Phi_R(u)) - Q G_2(\Phi_R(v)) \rVert_{L_{\mathcal P}^2(X)} &\,\leq\, 8 \gamma \tilde{C}_\phi M R^{\frac{p-1}{2}} T^{\beta} \lVert u - v \rVert_{L_{\mathcal{P}}^2(L^r(0,T;L^{p+1}_x))},
\end{align}
where the pathwise bounds \eqref{eq:growth4} and \eqref{eq:lip3} hold again $\P$-almost surely.
\end{subequations}

\emph{Step 3 (There exists $L,T\in (0,\infty)$ such that   $\Xi_{u_0}$ is a contraction).} 
For this last step, we introduce the cumulative constants $C_\Psi =\sum_{n=1}^NC_{\Psi_n}$ and  $\tilde{C}_\Psi =\sum_{n=1}^N\tilde{C}_{\Psi_n}$. 
Then,  we choose the constants from  \eqref{Eq79}  such that  
\begin{equation}\label{Eq31}
L \,\ge\, L_*(\gamma, u_0) \,=\, 4\gamma \|u_0 \|_{L_\Omega^2(H^1_x)},\qquad T\,\le\,  T_*(\gamma,p ,C_\Psi,C_\phi, M,R),
\end{equation}
where the latter constant is  determined by
\begin{equation*}
	\gamma M\bigl(
	C_{\Psi}T_*+ C_{\Psi}R^{p-1}T_*^\alpha +C_\phi T_*^{1/2}
	+C_\phi R^{\frac{p-1}{2}}T_*^\beta
	\bigr)\,=\, {1}/{16}.
\end{equation*}
Using the growth bounds \eqref{eq:growth1}, \eqref{eq:growth2}, \eqref{eq:growth3}, \eqref{eq:growth4} and \eqref{eq:growth5} from the  previous steps, we find then
\begin{align}
	\label{Eq33}&
	\| \Xi_{u_0} u\|_{L^2_{\mathcal{P}}(Y)} \,\\ &\quad\le\,2 
	\gamma \|u_0\|_{L_\Omega^2(H^1_x)}\,+\, 
	2\gamma M\bigl(
	C_{\Psi}T+ C_{\Psi}R^{p-1}T^\alpha +C_\phi T^{1/2}
	+C_\phi R^{\frac{p-1}{2}}T^\beta
	\bigr)\mathbb E^{1/2} \bigl[\|u\|_{L^\infty(0,T;H^1_x)\cap L^r(0,T;W_x^{1,p+1})}^2 \bigr]\\ &\quad \le\, L/2 \,+\, L/8,
\end{align}
and therefore $\Xi_{u_0}$ maps $\mathscr{M}_{L,T}$ into itself. We deduce  similarly
\begin{align}\begin{split}\label{Eq110}&
\| \Xi_{u_0} u -\ \Xi_{u_0} v\|_{L^2_{\mathcal{P}}(X)}
\\&\quad  \le\, 8\gamma M (\tilde{C}_\Psi  T+ \tilde{C}_\Psi R^{p-1}T^\alpha +\tilde{C}_\phi T^{1/2}
+\tilde{C}_\phi R^{\frac{p-1}{2}}T^\beta)\|u-v\|_{L^2_{\mathcal{P}}(C([0,T]; L^2_x)\cap L^r(0,T;L_x^{p+1}))}
\\&\quad  \le\, d_{\mathscr{M}_{L,T}}(u,v)/2,
\end{split}\end{align}
from
\eqref{eq:lip1}, \eqref{eq:lip2}, \eqref{eq:lip3} and \eqref{eq:lip4} and thus $ \Xi_{u_0}$ is a contraction as desired.

\emph{Step 4 (Conclusion).} We choose $L,T\in (0,\infty)$ as in \eqref{Eq31}, so that there exists a unique fixed point $u$ of 
$\Xi_{u_0}$ on $(\mathscr{M}_{L,T},d_{\mathscr{M}_{L,T}})$  by the previous step and Lemma \ref{lemma:metric_space}. 
Regarding its regularity, we deduce from \eqref{Eq33} that $
u\in L_{\mathcal{P}}^2( C([0,T];H^1_x))$. Using the Strichartz estimates from Theorem \ref{thm:hom_Strichartz} and Corollaries \ref{lemma:stochStrich}--\ref{cor:mult_strichartz} for an admissible pair $(\tilde{r},\tilde{q})$ in combination with the growth bounds \eqref{Eq25a}, \eqref{Eq26a}, \eqref{Eq27a} and \eqref{Eq27b} yields  analogously to the second step that  $u\in L^2_{\mathcal{P}}(L^{\tilde{r}}(0,T;W^{1,\tilde{q}}_x))$. This finishes the proof in the case $T_0 \le T$ and otherwise we can  iterate the above on time intervals $[kT, (k+1)T]$ for $k=1,\dots, \lceil T_0/T \rceil -1$ with initial value $u(kT)$, up to enlarging $L$. 
\end{proof}

\begin{corollary}[Truncated equation---uniqueness]\label{cor:trunc_uniq}
	Under the assumptions of Proposition \ref{prop:trunc_ex}, if there exists a stopping time $\sigma\le T_0$ and $v \in L^2_{\mathcal{P}}(C([0,\sigma];H^1_x) \cap L^r(0,\sigma;W^{1,p+1}_x))$  such that \eqref{eq:nlsmildtruncated} holds for all $t\in [0,\sigma]$ with $u$ replaced by $v$, $\P$-a.s., then we have $v= u$, $\P$-a.s., on $[0,\sigma]$.
\end{corollary}
    \begin{proof}
    	We let $T$ be as in \eqref{Eq31} and use that \eqref{eq:nlsmildtruncated} holds for $u$ as well as $v$ on $[0,\sigma]$, to estimate 
    	\begin{align*}
    		\| u-v \|_{L^2_{\mathcal{P}}(C([0,\sigma \wedge T];L^2_x)\cap L^r(0,\sigma \wedge T;L_x^{p+1}))} \,\le\, \| \Xi_{u_0 } ( \mathbf{1}_{[0,\sigma ] }u )- \Xi_{u_0 } ( \mathbf{1}_{[0,\sigma ] } v) \|_{L^2_{\mathcal{P}}(C([0, T];L^2_x)\cap L^r(0, T;L_x^{p+1}))}.
    	\end{align*}
    In light of the first estimate in \eqref{Eq110}, we can bound the right-hand side again by 
    \[
    \frac{1}{2} \| u-v \|_{L^2_{\mathcal{P}}(C([0,\sigma \wedge T];L^2_x)\cap L^r(0,\sigma \wedge T;L_x^{p+1}))}.
    \]
    It follows that $u = v$, $\P$-a.s., on $[0,\sigma \wedge T]$ and if $T_0>T$ we can iterate this procedure  on intervals $[\sigma \wedge (kT),  \sigma \wedge (k+1)T]$,  for $k=1,\dots, \lceil T_0/T \rceil -1$. 
    \end{proof}
\subsection{Transference to the original problem}\label{Sec:Transf_to_orig_prob}
We  turn our attention to the proof of Theorem \ref{Thm:WP_H1}, which is obtained from the results of the previous subsection together with suitable localization arguments.
\begin{proof}[Proof of Theorem \ref{Thm:WP_H1}]
    We also divide this proof into several steps.

    \emph{Step 1 (Reduction and setup).}
    Firstly, we observe that if $u_0\notin L_\Omega^2(H^1_x)$, we can pass to the equivalent probability law  ${\mathbb{Q}}\propto \exp(-\|u_0\|_{H^1_x})\mathbb{P}$, with respect to which $\|u_0\|_{H^1_x}$ admits all moments. As the density function is $\mathcal{F}_0$-measurable, the property of being a local mild solution to \eqref{eq:general_SNLS} as defined in Definition \ref{defi_local_mild_sol} is invariant under this change of measure. Hence, we can assume without loss of generality that $u_0\in L_\Omega^2(H^1_x)$.
   By Proposition \ref{prop:trunc_ex},
    there exists then for every $R, T_0\in \N$ a process $u_{R,T_0} \in L^2_{\mathcal{P}}(C([0,T_0];H^1_x) \cap L^r(0,T_0;W^{1,p+1}_x))$ satisfying, $\P$-a.s., \eqref{eq:nlsmildtruncated}  for $t\in [0,T_0]$. By Corollary \ref{cor:trunc_uniq} we see that these solutions extend each other as $T_0$ increases so that we can define 
    $u_R$ as the process which coincides with $u_{R,T_0}$, $\P$-a.s., on $[0,T_0]$ for any $T_0\in \N$. Then we set 
    \begin{equation}
    \label{eq:deftauR}
    \tau_{R} \,=\, \inf\Bigl\{
t\in [0,\infty) \,\Big|\, \|u_{R}\|_{C([0,t]; L^{p+1}_x)} \ge R
\Bigr\},
    \end{equation}
	with
    the conventions $\inf \emptyset =\infty$  and  $[0,\tau_R]= [0,\infty)$ on $\{\tau_R=\infty\}$.

    \emph{Step 2 (Defining the candidate solution).} Now let $\tilde{R}<R$, so that
    $u_{\tilde{R}}$ satisfies also \eqref{eq:nlsmildtruncated}, $\P$-a.s., on $[0,\tau_{\tilde{R}}]$ and hence we have $u_{\tilde{R}} = u_R$ on $[0,\tau_{\tilde{R}}]$, $\P$-a.s., in light of Corollary \ref{cor:trunc_uniq}. As a result, we deduce moreover $\tau_R \ge \tau_{\tilde{R}}$, $\P$-a.s.,  and using the continuity  of $u_R$
    in $L^{p+1}_x$ due to \eqref{eq:Sobolev_embedding} even $\tau_R>\tau_{\tilde{R}}$ on the set $\{\|u_0\|_{L^{p+1}_x}<R\}$. Thus we can define 
    \begin{equation}\label{Eq34}
        \tau \,=\, \lim_{R\to\infty}\tau_R,\qquad  u\,=\, u_R \text{ on }
        [0,\tau_R],
    \end{equation}
$\P$-a.s.,
    and because the events  $\{\|u_0\|_{L^{p+1}_x}<R\}$ exhaust $\Omega$
    it holds
    \begin{equation}
            \label{Eq32} \P\biggl(
    \bigcup_{R\in \N}[0,\tau_R] \,= \, [0,\tau) \biggr) \,=\, 1.
        \end{equation}
    The latter implies in particular that $\tau>0$, $\mathbb{P}$-almost surely.

    \emph{Step 3 ($(u,\tau)$ is the maximal, unique local mild solution to \eqref{eq:general_SNLS}).} 
    We first observe that $u\in C([0,\tau); H^1_x)\cap L_{\mathrm{loc}}^r([0,\tau);W_x^{1,p+1})$, $\P$-a.s., since it inherits the path-regularity  of the processes $u_R$.
    Furthermore, on the interval $[0,\tau_R]$ the truncated version  \eqref{eq:nlsmildtruncated} for $u_R$ is equivalent to \eqref{eq:mild_solution_SNLS} and, as $u=u_R$ on $[0,\tau_R]$, we deduce that \eqref{eq:mild_solution_SNLS} holds $\P$-a.s.\ on $[0,\tau )$ by \eqref{Eq32}.
    Therefore, we conclude that $(u,\tau)$ is a local mild solution to \eqref{eq:general_SNLS}.
    Let $(v,\sigma)$ be another local mild solution. Then we define for $L,R,T_0\in \N$ stopping times
    \[
    \sigma_{(L,R,T_0)} \,=\, \inf\Bigl\{
    t\in [0,\sigma)\,\Big|\,  
    \|v\|_{C([0,t]; H^1_x)\cap L^r(0,t;W_x^{1,p+1})}\ge L 
    \text{ or }
    \|v\|_{C([0,t]; L^{p+1}_x)}\ge R 
    \Bigr\}  \wedge (\sigma \wedge  T_0),
    \]
    so that $   \sigma_{(L,R,T_0)} \to \sigma $ as $L,R,T_0\to\infty$, $\mathbb{P}$-almost surely. The process 
    \[
    \tilde{v} \,=\,
        v\mathbf{1}_{\{\sigma>0\}}\,+\,u_0\mathbf{1}_{\{\sigma =0\}}
    \]
    lies consequently in $L^2_{\mathcal{P}}(C([0,\sigma_{(L,R,T_0)}];H^1_x) \cap L^r(0,\sigma_{(L,R,T_0)};W^{1,p+1}_x))$ and is a solution to the truncated mild formulation \eqref{eq:nlsmildtruncated}. Whence, using Corollary \ref{cor:trunc_uniq}, we deduce that, $\P$-a.s., $\tilde{v} = u_R$ on $[0, \sigma_{(L,R,T_0)}]$ and  $\sigma_{(L,R,T_0)} \le \tau_R$. By \eqref{Eq32} it follows that $\sigma_{(L,R,T_0)}\le \tau$ and by \eqref{Eq34} also $\tilde{v}
    =u$ on $[0, \sigma_{(L,R,T_0)}]$, $\P$-almost surely. By taking the limits $L,R,T_0\to\infty$ we conclude that, $\P$-a.s., $\sigma \le \tau$ and $v = \tilde{v} = u $ on $[0,\sigma )$. In other words:  $(u,\tau)$ is maximal and unique.

    \emph{Step 4 (Blow-up alternative and regularity).}
    We use the second identity from \eqref{Eq34} combined with $\|u_R(\tau_R)\|_{L_x^{p+1}} = R$ on $\{0<\tau_R<\infty \}$
    to conclude that, $\P$-a.s.,
    \[
    {\tau}_{R} \,=\, \inf\Bigl\{
t\in [0,\tau) \,\Big|\, \|u\|_{C([0,t]; L^{p+1}_x)} \ge R
\Bigr\} .
    \]
    Hence, $\P$-a.s., we  have $\|u(\tau_R)\|_{L_x^{p+1}} \ge R$ on $\{\tau_R<\infty \}$ and on the latter set also $\tau_R<\tau$ by \eqref{Eq32}. Consequently, the  blow-up criterion \eqref{eq:blow_up} follows from the first identity of \eqref{Eq34}. The assertion regarding the additional regularity of $u$ is a consequence of the additional regularity of $u_R$ obtained in Proposition \ref{prop:trunc_ex}. 
\end{proof}
\section{A priori estimates}\label{Sec_APE}
In the previous section, we constructed a local maximal solution for any $H^1_x$-subcritical stochastic nonlinear Schr\"odinger equation, and showed that it satisfies the blow-up criterion~\eqref{eq:blow_up}.
In this section, we will show that in the $L^2_x$-subcritical case (meaning that $p$ in Assumption~\ref{ass:nonlinear} additionally satisfies $p < 1+4/d$), a blow-up criterion can be formulated in terms of only the $L^2_x$-norm. Our statement can be compared to~\cite[Proposition 8]{hornung_SNLS}, which treats power-type nonlinearities as in~\eqref{Eq100} and imposes 
the additional relation \eqref{Eq36} between $p$ and $q$.
Since we can cover the full subcritical range for both $p$ and $q$, our proof shows that the rightmost condition of \eqref{Eq36} is not actually needed to have global well-posedness.

Our precise statement is as follows.

\begin{proposition}[Blow-up criterion in $L^2$-subcritical case]
\label{prop:global}
     Suppose Assumption \ref{ass:nonlinear} holds, where $p$ additionally satisfies 
    \begin{equation}
    \label{eq:mass_subcrit}
        p\,\in \, (1, 1+4/d).
    \end{equation}
     Let $(u,\tau)$ be the maximal, unique local mild solution to \eqref{eq:general_SNLS} provided by Theorem \ref{Thm:WP_H1}.
     Then we have the blow-up alternative
	\begin{equation}\label{eq:blow_up_improv}
		\PP[\Big]{\tau < \infty,\, \sup_{t\in [0, \tau)}\norm{u(t)}_{L^{2}_x} <\infty } \,=\, 0.
	\end{equation}
\end{proposition}

The proof proceeds in two steps.
First, we show in Proposition~\ref{prop:reduce_to_L2} that non-blowup in $L^2_x$ implies non-blowup of the $L^r(L^{p+1}_x)$-norm, by providing a new a priori estimate.
Its proof makes efficient use of stopping times to partition the $L^r(L^{p+1}_x)$-norm into equally sized pieces.
This is a generalization and strengthening of the argument used in~\cite[Proposition 4.5]{gnann2024solitary} (which provides almost sure boundedness instead of integrability in $\omega$).

Secondly, we show in Proposition~\ref{prop:reduce_to_strichartz} that non-blowup of the $L^r(L^{p+1}_x)$-norm implies non-blowup of the $H^1_x$-norm.
There, the proof is more straightforward and  akin to a standard Gr\"onwall argument. Together, they imply the main result of the current section.

\begin{proposition}[$L^2$-boundedness implies boundedness in Strichartz spaces]\label{prop:reduce_to_L2}
    Suppose that Assumption~\ref{ass:nonlinear} holds, where $p$ additionally satisfies \eqref{eq:mass_subcrit}
    and let $r$ be such that $(r,p+1)$ is admissible.
    Let $(u,\tau)$ be the maximal, unique local mild solution to \eqref{eq:general_SNLS} provided by Theorem \ref{Thm:WP_H1}.
    For $K \in (0,\infty)$, define the stopping time
    \begin{equation} 
    \label{eq:L2stopping}
        \sigma_K \coloneq \, \inf \cur{t \in [0,\tau) : \norm{u(t)}_{L^2_x} > K} \wedge \tau.
    \end{equation}
    Then for any $T <\infty$ and $s \in (0,r)$, it holds that
    \begin{equation}
        \label{eq:strichartzAPest}
        \EE[\Big]{\norm{u}_{L^r(0,\sigma_K \wedge T;L^{p+1}_x)}^s} < \infty.
    \end{equation}
\end{proposition}
\begin{proof}
    We observe that to show~\eqref{eq:strichartzAPest} it suffices to find $T^* > 0$ such that the following tail bound holds uniformly in $\lambda$:
    \begin{equation}\label{eq:Strichartz_tailbound}
    \PP[\Big]{ \norm{u}_{L^r(0,\sigma_{K} \wedge T^*;L^{p+1}_x)}^r \,>\, \lambda }\,\lesssim \,\lambda^{-1}.
    \end{equation}
    Indeed, given $s \in (0,r)$ the layer cake representation gives
    \begin{align*}
          \EE[\Big]{\norm{u}_{L^r(0,\sigma_{K} \wedge T^*;L^{p+1}_x)}^s} \,=\, s/r \int_0^\infty
          \lambda^{s/r-1} 
          \PP[\Big]{ \norm{u}_{L^r(0,\sigma_{K} \wedge T^*;L^{p+1}_x)}^r \,>\, \lambda }\,\d \lambda < \infty,
    \end{align*}
    since $s / r < 1$.
    Iterating this estimate on the time intervals 
    $[\sigma_K\wedge  (l-1)T^*, \sigma_K\wedge  l T^* ]$ for $l = 1 \ldots \lceil T / T^* \rceil$, it follows that~\eqref{eq:strichartzAPest} holds for arbitrary $T > 0$.
    
    To verify \eqref{eq:Strichartz_tailbound} we first observe that the pair $(2+4/d, 2+4/d)$ is admissible.
    Since $(r,p+1)$ is admissible by assumption and $p+1 < 2+4/d$ by~\eqref{eq:mass_subcrit}, it then follows from~\eqref{eq:admissible} that $r > p+1$.
    As a consequence, after defining
    \begin{equation}\label{Eq103}
        \theta \coloneq 1/r' \, - \,  p/r\, = 1 - (p+1)/r,
    \end{equation}
    we find that $\theta \in (0,1)$.
    For fixed $T > 0$, we now define a recursive sequence of stopping times starting from $\rho_0 = 0$ by
    \begin{align*}
        \rho_k \,\coloneq\, \sup \cur[\Big]{t \in [\rho_{k-1},\sigma_{K} \wedge T] : \norm{u}^r_{L^{r}(\rho_{k-1},t;L^{p+1}_x)} \leq (3\gamma K)^r }, \qquad k \in \bbN,
    \end{align*}
    where $\gamma<\infty$ is the constant from Theorems \ref{thm:hom_Strichartz}--\ref{thm:stoch_Strichartz}.
    Fixing now $k_0 \in \bbN$, we see that on the event 
    \begin{equation*}
    A_{k_0} = \cur[\Big]{\norm{u}_{L^r(0,\sigma_{K} \wedge T;L^{p+1}_x)}^r \geq k_0 (3\gamma K)^r}
    \end{equation*}
     we have  $\rho_k > \rho_{k-1}$ and $\norm{u}_{L^r(\rho_{k-1},\rho_k;L^{p+1}_x)} = 3\gamma K$ for every $k \in \cur{1, \ldots, k_0}$, as follows from the definition of $\rho_k$ and the dominated convergence theorem.
    For such $k$ it then also holds $\bbP$-a.s.\ that for all $t \in [\rho_{k-1},\rho_{k}]$:
    \begin{equation}\label{eq:mildstoppingblowup}
		u(t) = S(t - \rho_{k-1})u(\rho_{k-1}) + \int_{\rho_{k-1}}^t S(t-t')F(\cdot, u(t')) \d t' + \int_{\rho_{k-1}}^t S(t-t')G(u(t')) \d W(t').
	\end{equation}
    We will use this information to deduce an upper bound on the probability that the event $A_{k_0}$ occurs.
    To do this, we first estimate the middle term of~\eqref{eq:mildstoppingblowup} as:
    \begin{align*}
        & \norm[\Big]{t \mapsto \int_{\rho_{k-1}}^{t} S(\cdot -t')F(\cdot, u(t')) \d t'}_{L^r(\rho_{k-1},\rho_{k};L^{p+1}_x)} \\
        &\quad\leq \gamma \tilde{C}_{\Psi} \sum_{n=1}^N \bra[\Big]{ \norm{F_1^{(n)}(u)}_{L^1(\rho_{k-1},\rho_k;L^2_x)} + \norm{F_2^{(n)}(u)}_{L^{r'}(\rho_{k-1},\rho_{k};L^{1+1/p}_x)}} \\
        &\quad\leq \gamma \tilde{C}_{\Psi} N M \bigl( \norm{u}_{L^1(\rho_{k-1},\rho_k;L^2_x)} + \norm{\norm{u}^{p}_{L^{p+1}_x}}_{L^{r'}(\rho_{k-1},\rho_{k})} \bigr)\\
        &\quad\leq \gamma \tilde{C}_{\Psi} N M \bigl(T \norm{u}_{L^{\infty}(\rho_{k-1},\rho_k;L^2_x)} + T^{\theta}\norm{\norm{u}^{p}_{L^{p+1}_x}}_{L^{r/p}(\rho_{k-1},\rho_{k})} \bigr)\\
        &\quad\leq \gamma \tilde{C}_{\Psi} N M \bigl(T K + T^{\theta}(3\gamma K)^p\bigr),
    \end{align*}
    where we have used the decompositions~\eqref{eqn_f_tilda_equals} and~\eqref{Eq29} together with Corollary \ref{cor:mult_strichartz}, the estimates~\eqref{eq:Lipschtiz_est_F1} and~\eqref{eq:F2_lip} with $v = 0$, H\"older's inequality with exponents given by~\eqref{Eq103}, and finally the definition of $\sigma_K$ and $\rho_k$. The occurring $\Psi$-dependent constant is moreover the one of \eqref{constants_dep_on_psi}.
    Thus, taking the $L^r(L^{p+1}_x)$-norm of~\eqref{eq:mildstoppingblowup} and using the above estimate shows that, on the set $A_{k_0}$ and for all $k \in \cur{1, \ldots, k_0}$, we have:
    \begin{equation}
    \label{eq:L2blowupstrichest_new}
        3\gamma K\, \leq \, \gamma \norm{u(\rho_{k-1})}_{L^2_x}
        \,+\, \gamma \tilde{C}_{\Psi} N M \bigl(T K + T^{\theta}(3\gamma K)^p \bigr)+ S_k,
    \end{equation}
    where we used Theorem \ref{thm:hom_Strichartz} and abbreviated
    \begin{align*}
        S_{k} &= \norm[\Big]{\int_{\rho_{k-1}}^{\cdot} S({\cdot}-t')G(u(t'))\d W(t')}_{L^r(\rho_{k-1},\rho_{k};L^{p+1}_x)}, \qquad k \in \bbN.
    \end{align*}
    Thus, choosing $T=T^*$ small enough (based on $\gamma$, $\tilde{C}_{\Psi}$, $N$, $M$, $K$, $p$) and noting that $\norm{u(\rho_{k-1})}_{L^2_x} \leq K$, it follows from \eqref{eq:L2blowupstrichest_new} that
    \begin{equation}
    \label{eq:L2blowupPPest_new}
    \begin{aligned}
        &\PP[\Big]{\norm{u}_{L^r(0,\sigma_K \wedge T^*;L^{p+1}_x)}^r \geq k_0 (3\gamma K)^r} \,=\, \PP{A_{k_0}}
        \,\leq\, \PP[\big]{S_k \geq \gamma K \text{ for all } k \in \cur{1, \ldots, k_0}} \\
        &\qquad\leq \, \PP[\bigg]{\sum_{k=1}^{\infty} S_k^{\zeta} \geq k_0(\gamma K)^{\zeta}} 
        \,\leq \, k^{-1}_0\cdot \biggl( (\gamma K)^{-\zeta}\sum_{k=1}^{\infty}\EE{S^{\zeta}_k}\biggr),
    \end{aligned}
    \end{equation}
    for any $\zeta > 0$ by the Chebyshev--Markov inequality.

    It now only remains to bound $S_k$ appropriately.
    Similarly to how we bounded $F$, we first deduce from the estimates~\eqref{eq:Lipschtiz_est_G1} and~\eqref{eq:G2_lip} (with $v = 0$) and the definition of $\theta$, $\sigma_K$ and $\rho_k$ that
    \begin{align*}
        \norm{G_1(u)}_{L^2(\rho_{k-1},\rho_{k};L^2_x)}\,&\le\, (\rho_{k} \,-\,\rho_{k-1})^{1/2} M K,
        \\
        \norm{G_2(u)}_{L^2(\rho_{k-1},\rho_{k};L^2_x)} &\le\, (\rho_{k} \,-\,\rho_{k-1})^{\theta/2} M (3\gamma K)^{(p+1)/2}.
    \end{align*}
    Hence, taking $\zeta = 2/\theta$ (note that $\zeta \in (2,\infty)$ since $\theta \in (0,1)$) and using the stochastic Strichartz estimate from Corollary \ref{lemma:stochStrich}, we finally see that
    \begin{align*}
        \EE{S_k^\zeta} \, &\leq \,(2\gamma \sqrt{\zeta} \tilde{C}_\phi)^\zeta \Bigl( \EE{\norm{G_1(u)}^\zeta_{L^2(\rho_{k-1},\rho_{k};L^2_x)}} \,
        + \,\EE{\norm{G_2(u)}^\zeta_{L^2(\rho_{k-1},\rho_{k};L^2_x)}} \Bigr) \\
        &\leq\, (2\gamma \sqrt{\zeta} \tilde{C}_\phi)^\zeta \Bigl((MK)^{\zeta} {(T^*)}^{\zeta/2 - 1} + M^\zeta (3\gamma K)^{\zeta(p+1)/2} \Bigr)\EE[\big]{(\rho_{k} - \rho_{k-1})},
    \end{align*}
    where we recall the notation $\tilde{C}_\phi$ was defined in \eqref{constants_dep_on_phi}.
    Since $\rho_k \leq T^*$ by definition, it follows from a telescoping series argument that $\EE{S_k^\zeta}$ is summable in $k$.
    Hence, the infinite sum in~\eqref{eq:L2blowupPPest_new} converges to a value independent of $k_0$, so that~\eqref{eq:Strichartz_tailbound} holds.
\end{proof}

\begin{proposition}[Boundedness in Strichartz spaces implies $H^1_x$-boundedness]
\label{prop:reduce_to_strichartz}
In the setting of Proposition~\ref{prop:reduce_to_L2}, define for $K \in (0,\infty)$ the stopping time
\begin{equation}
    \tilde{\sigma}_K = \sup \cur[\big]{t \in [0,\tau) : \norm{u}_{L^r(0,t;L^{p+1}_x)} \leq K}.
\end{equation}
Then there exist constants $L_K, G_K<\infty$, such that the following estimate holds for all $T <\infty$:
\begin{equation}\label{Eq107}
\EE[\bigg]{\sup_{t \in [0,\tilde{\sigma}_K \wedge T) }\norm{u(t)}_{H_x^1}^2} \,\le\, L_K e^{G_K T} \EE[\big]{\|u_0 \|_{H_x^1}^2}.
\end{equation}

\end{proposition}
\begin{proof}
Fix $K ,T <\infty$. In the following, we replace $\tilde{\sigma}_K$ by
\[
\tilde{\sigma}_{K} \wedge \sup\Bigl\{
t \in [0,\tilde \sigma_K] \Big| \|u\|_{C([0,t];H^1_x)} \,+\,
        \|u\|_{L^r(0, t ; W^{1,p+1}_x) } \le R
\Bigr\}
\]
for $R<\infty$ and remark that at the end, one can obtain the same assertion for the original $\tilde{\sigma}_K$ by Fatou's lemma.

Then, $u$ is defined on $[0,\tilde{\sigma}_K \wedge T]$ and satisfies the mild formulation~\eqref{eq:mild_solution_SNLS}.
    Employing the deterministic and stochastic Strichartz estimates from 
   Theorem \ref{thm:hom_Strichartz} and Corollaries \ref{lemma:stochStrich}--\ref{cor:mult_strichartz}, together with the decompositions~\eqref{eqn_f_tilda_equals} and~\eqref{Eq29}, we find that
    \begin{equation}
    \label{eq:H1blowupBigest}
    \begin{aligned}
        &\|u\|_{L^2_{\mathcal P}(C([0,\tilde{\sigma}_K \wedge T];H^1_x))} \,+\,
        \|u\|_{L^2_{\mathcal P}(L^r(0, \tilde{\sigma}_K \wedge T ; W^{1,p+1}_x) )}
        \\&\quad 
        \le 2\gamma\Bigl(
        \|u_0\|_{L^2_{\Omega} ( H^1_x)} \,+\, C_{\Psi}\sum_{n=1}^N \bra[\big]{ \|F_1^{(n)}(u)\|_{L^2_{\mathcal P}(L^1(0, \tilde{\sigma}_K \wedge T ; H^1_x) )} \,+\, \|F_2^{(n)}(u)\|_{L^2_{\mathcal P}(L^{r'}(0, \tilde{\sigma}_K \wedge T ; W^{1,1+1/p}_x) )}}
        \Bigr) \\
        &\qquad +\, \sqrt{8}\gamma C_{\phi} 
        \Bigl(\|G_1(u)\|_{L^2_{\mathcal P}(L^2(0, \tilde{\sigma}_K \wedge T ; H^1_x) )} \,+\, \|G_2(u)\|_{L^2_{\mathcal P}(L^{2}(0, \tilde{\sigma}_K \wedge T; H^{1}_x) )}
        \Bigr).
    \end{aligned}
    \end{equation}
    We proceed to bound two of the terms on the right-hand side by 
    \begin{align*}
        \|F_1^{(n)}(u)\|_{L^1(0, \tilde{\sigma}_K \wedge T ; H^1_x) }\,&\le\,  M\|u\|_{L^1(0, \tilde{\sigma}_K \wedge T ; H^1_x) }\,\le\, MT\|u\|_{L^\infty(0, \tilde{\sigma}_K \wedge T ; H^1_x) },
        \\
        \|G_1(u)\|_{L^2(0, \tilde{\sigma}_K \wedge T ; H^1_x) } \,&\le\, 
        M \|u\|_{L^2(0, \tilde{\sigma}_K \wedge T ; H^1_x) } \,\le\, M \sqrt{T}\|u\|_{L^\infty(0, \tilde{\sigma}_K \wedge T ; H^1_x) } ,
        \end{align*}
    using \eqref{eq:bddness_F1} and \eqref{eq:bddness_G1}.
    Applying  \eqref{eq:F2_bound} and \eqref{eq:G2_bound} and then H\"older's inequality while recalling the definition \eqref{Eq103} of $\theta$  we bound additionally
    \begin{align*}
        \|F_2^{(n)}(u)\|_{L^{r'}(0, \tilde{\sigma}_K \wedge T ; W^{1,1+1/p}_x) }\,&\le\, 
        M T^\theta \|u\|_{L^r(0, \tilde{\sigma}_K \wedge T ; L^{p+1}_x)}^{p-1}
        \|u\|_{L^r(0,\tilde{\sigma}_K \wedge T ; W^{1,p+1}_x)}
        \\
        &\le\, 
        M K^{p-1} T^\theta
        \|u\|_{L^r(0, \tilde{\sigma}_K \wedge T ; W^{1,p+1}_x)},
        \\
         \|G_2(u)\|_{L^2(0, \tilde{\sigma}_K \wedge T ; H^1_x) }\,&\le\, 
        M T^{\theta/2} \|u\|_{L^r(0, \tilde{\sigma}_K \wedge T ; L^{p+1}_x)}^{(p-1)/2}
        \|u\|_{L^r(0, \tilde{\sigma}_K \wedge T ; W^{1,p+1}_x)}
        \\
        &\le\,   M K^{(p-1)/2}T^{\theta/2} 
        \|u\|_{L^r(0, \tilde{\sigma}_K \wedge T ; W^{1,p+1}_x)}.
    \end{align*}
    Thus, taking $T^*$ small enough (based on $\gamma$, $p$, $M$, $N$, $K$, $C_{\phi}$, $C_{\Psi}$),
    we may use the above estimates to absorb every term of the right-hand side of~\eqref{eq:H1blowupBigest} except the first and find
    \[
        \|u\|_{L^2_{\mathcal P}(C([0,\tilde{\sigma}_K \wedge T^*]; H^1_x) )}
    \,\le\, 4\gamma \|u_0\|_{L^2_\Omega ( H^1_x) }.
    \]
    Iterating this estimate on the time intervals 
    $[\tilde{\sigma}_K\wedge  (l-1)T^*, \tilde{\sigma}_K\wedge  l T^* ]$ for $l = 1 \ldots \lceil T / T^* \rceil$, we obtain the stated estimate.
\end{proof}
Using the estimates provided by Propositions~\ref{prop:reduce_to_L2} and~\ref{prop:reduce_to_strichartz}, we can prove Proposition~\ref{prop:global} straightforwardly.
\begin{proof}[Proof of Proposition~\ref{prop:global}]
    For $R > 0$, consider the events $A = \cur{ \tau < \infty}$ and $B_R =\{ \|u_0 \|_{H^1_x} \le  R \}\in  \cF_0$. 
    By the blow-up condition \eqref{eq:blow_up}, it holds $\P$-a.s.\ on the event $A$ that $\sup_{t \in [0,\tau)}\norm{u(t)}_{L^{p+1}_x} = \infty$, which by Sobolev embedding \eqref{eq:Sobolev_embedding} implies $\sup_{t \in [0,\tau)}\norm{u(t)}_{H^1_x} = \infty$.
    From Proposition~\ref{prop:reduce_to_strichartz} it follows that on the event $A \cap B_R$ it holds $\P$-a.s.\ that $\tau > \tilde{\sigma}_K$ for any $K > 0$, and consequently  $\sup_{t \in [0,\tau)}\norm{u}_{L^r(0,t;L^{p+1}_x)} = \infty$.
    Similarly, we deduce from Proposition~\ref{prop:reduce_to_L2} that on the event $A \cap B_R$ it holds  $\tau > \sigma_K$, $\P$-a.s., for every $K > 0$, and therefore $\sup_{t \in [0,\tau)} \norm{u(t)}_{L^2_x} = \infty$.
    Overall, we find for any $R > 0$ that
    \begin{equation*}
        \PP[\Big]{\tau < \infty,\, \sup_{t \in [0,\tau)} \norm{u(t)}_{L^2_x} < \infty,\, B_R} = 0.
    \end{equation*}
    Taking $R \to \infty$ while using $\lim_{R \to \infty}\PP{B_R} = 1$, the result follows.
\end{proof}

\section{Conservation laws}\label{Sec_quant}

Theorem \ref{Thm:WP_H1}  provides  the assertions of Theorem \ref{thm_intro} regarding the local well-posedness of \eqref{Eq100}. Indeed, following the comments at the beginning of Section \ref{sec_local}, its It\^o formulation \eqref{eq:snls_ito} takes precisely the form \eqref{eq:general_SNLS} and the regularity imposed on the noise coefficients in \eqref{noise_reg_intro}--\eqref{noise_reg_intro_extra} guarantees precisely the one of \eqref{eq:psi_correction} required in Assumption \ref{ass:nonlinear}.  
In the current section, we provide the final ingredients to conclude together with Proposition \ref{prop:global} the global well-posedness of \eqref{Eq100} in the special situations \ref{item_A} and \ref{item_B}. We remark that in both, the noise coefficients $(\phi_k)_{k\in\N} $ are assumed to be real, and in particular
\begin{equation}
    \Psi = -\tfrac{1}{2}\sum_{k \in \bbN} \phi_k^2.
\end{equation}
Moreover, by the blow-up criteria from Theorem \ref{Thm:WP_H1} and Proposition \ref{prop:global}, this amounts to obtain upper bounds on the $L^{p+1}_x$-norm and the $L^2_x$-norm of solutions to \eqref{Eq100}, respectively. The latter are usually obtained through stochastic variants of the conservation laws for  mass and energy \eqref{eqn_conserved_qtts} of the deterministic Schr\"odinger equation. Since such stochastic versions have been established in several preceding works, see for instance \cite{DeBouard_Debussche_H1,gnann2024solitary,hornung_SNLS}, we keep their presentation brief. 

Firstly, a formal application of It\^o's formula shows that also in the stochastic case the mass is conserved along the trajectories of the solution. This can, e.g.,  be justified  using that up to stopping times approaching its maximal time of existence, $u$ is a solution to the truncated equation \eqref{eq:nlsmildtruncated} and then applying the mild It\^o formula \cite[Theorem 1]{mild}  to the latter, as is done in \cite[Proposition 4.6]{gnann2024solitary} for parametrically forced stochastic nonlinear Schr\"odinger equations. The result is the following.

\begin{proposition}[Conservation of mass]\label{prop:mass_cons}
Suppose that the assumptions of Theorem \ref{thm_intro} hold and that $(\phi_k)_{k\in \N}$  is real-valued. Then, the maximal unique local solution to \eqref{Eq100} satisfies 
	\[
	\|u(t)\|_{L^2_x}^2 \,=\, \|u_0\|_{L^2_x}^2,\qquad t\in [0,\tau),
	\]
	$\P$-almost surely. 
\end{proposition}

While the above suffices to deduce global well-posedness in the mass-subcritical regime, i.e., under the assumption \ref{item_B}, in which Proposition \ref{prop:global} applies, an estimate on the energy 
\begin{align}\label{eqn_energy_def}
E(u)\, = \,\int_{\R^d} \frac{|\nabla u|^2}{2}  - \frac{\lambda |u|^{p+1} }{p+1}\, \d x,
\end{align}
is needed otherwise. Regarding this, we restrict ourselves to the case of linear noise, i.e., we impose that $q=1$, since otherwise the It\^o expansion of $E$ does at least not straightforwardly lead to a closed a priori estimate. Then, since the local $H^1_x$-valued solution to \eqref{Eq100} from Theorem \ref{Thm:WP_H1} has the same path-regularity as the ones constructed in \cite[Theorem 4.1]{DeBouard_Debussche_H1}, the following properties are derived in the same way as in \cite[Section 4.2]{DeBouard_Debussche_H1}: Firstly, a regularization procedure allows to justify It\^o's formula to compute the evolution of $E(u)$ stated in \cite[Proposition 4.5]{DeBouard_Debussche_H1}, in which $\overline{u}$ denotes the complex conjugate of $u$.

\begin{proposition}
Suppose that the assumptions of Theorem \ref{thm_intro} hold, that $(\phi_k)_{k\in \N}$  are real-valued and that the noise is linear multiplicative, i.e., that $q=1$. Then, the maximal unique local solution to \eqref{Eq100} satisfies 
	\[
	E(u(t)) \,=\, E(u_0)\,+\frac{1}2 \sum_{k\in \N} \int_0^t \| u \nabla \phi_k \|_{L^2_x}^2\,  \d s\, +\, \Imag \sum_{k\in \N}
     \int_0^t  \int_{\R^d} \overline{u} \nabla u \cdot \nabla \phi_k \, \d x \, \d \beta_k(s)
    ,\qquad t\in [0,\tau),
	\]
	$\P$-almost surely. 
\end{proposition}
In the situation  \ref{item_A} in which additionally $\lambda \le 0$, due to the regularity of $(\phi_k)_{k\in \N}$ imposed in \eqref{noise_reg_intro}--\eqref{noise_reg_intro_extra} and the already established Proposition \ref{prop:mass_cons}, this allows to obtain the following moment bound on the energy from \cite[Theorem 4.6]{DeBouard_Debussche_H1}  by an application of the Burkholder--Davis--Gundy inequality.
\begin{proposition}[Moment bound on energy]\label{prop:energy_bound}
    Suppose that the assumptions of Theorem \ref{thm_intro} hold, that $(\phi_k)_{k\in \N}$  are real-valued and that we are in the situation \ref{item_A}. Then, for all $K,T \in (0,\infty)$, it holds
  \begin{align}\label{eqn_estimate}
      \mathbb E \Bigl[\sup_{ 0 \le t < \tau \wedge T } E(u(t)) \Big| \|u_0\|_{L^2_x}^2 + E(u_0) \le K \Bigr]  \,\le\, C,
  \end{align}
  for a constant $C$ depending only on $(\phi_k)_{k\in \N} $, $K$   and $T$. 
\end{proposition}
 Finally, since by $\lambda \le 0$, \eqref{eqn_estimate} implies a control on both terms of \eqref{eqn_energy_def}  separately, the above  allows to check the blow-up condition
\eqref{eq:blow_up}, as desired.

\vspace{.2cm}
\noindent
\textbf{Data availability.} This manuscript has no associated data.

\noindent
\textbf{Declaration – Conflict of interest.} The authors have no conflict of interest.

\noindent
\textbf{AI disclosure statement.}
The large language model GPT-5.6 by OpenAI was used for proofreading and checking the correctness of the manuscript.
The authors take full responsibility for the contents of the manuscript.

\end{document}